\documentclass[a4paper]{amsart}
\usepackage[active]{srcltx}
\usepackage[all]{xy}
\usepackage{subfig}
\usepackage{hyperref}
\usepackage{mathrsfs}

\usepackage{esint}
\usepackage{amsmath}
\usepackage{amssymb,latexsym}
\usepackage{mathrsfs}
\usepackage{graphics}
\usepackage{latexsym}
\usepackage{psfrag}
\usepackage{import}
\usepackage{verbatim}
\usepackage{graphicx}
\usepackage[usenames]{color}
\usepackage{pifont,marvosym}

\theoremstyle{plain}
\newtheorem{lemma}{Lemma}[section]
\newtheorem{theorem}[lemma]{Theorem}
\newtheorem{proposition}[lemma]{Proposition}
\newtheorem{corollary}[lemma]{Corollary}

\theoremstyle{definition}

\newtheorem{definition}[lemma]{Definition}
\newtheorem{remark}[lemma]{Remark}

\numberwithin{equation}{section}

\newcommand{\dom}{\textrm{Dom\,}}

\newcommand{\R}{\mathbb{R}}
\newcommand{\N}{\mathbb{N}}

\newcommand{\supp}{\text{\rm supp}}

\newcommand{\Lip}{\mathrm{Lip}}
\newcommand{\gr}{\textrm{graph}}

\newcommand{\ve}{\varepsilon}

\newcommand{\f}{\varphi}

\renewcommand{\r}{\varrho}

\renewcommand{\L}{\mathcal{L}}

\newcommand{\mm}{\mathfrak m}
\newcommand{\Tan}{{\rm Tan}}
\newcommand{\sfd}{\mathsf d}

\begin{document}
\title[Tangent lines and Lipschitz differentiability spaces]
{Tangent lines and \\ Lipschitz differentiability spaces}
\author{Fabio Cavalletti}
\address{Centro de Giorgi - SNS}
\email{fabio.cavalletti@sns.it}

\author{Tapio Rajala}
\address{University of Jyvaskyla\\
         Department of Mathematics and Statistics \\
         P.O. Box 35 (MaD) \\
         FI-40014 University of Jyvaskyla \\
         Finland}
\email{tapio.m.rajala@jyu.fi}

%\keywords{optimal transport; existence of maps; uniqueness of maps; measure contraction property}
%

\bibliographystyle{plain}

\begin{abstract}
We study the existence of tangent lines, i.e. subsets of the tangent space isometric to the real line,
in tangent spaces of metric spaces.
We first revisit the almost everywhere metric differentiability of Lipschitz continuous curves. 
We then show that any blow-up done at a point of metric differentiability and of density
one for the domain of the curve gives a tangent line.   

Metric differentiability enjoys a Borel measurability property and this will permit us to use it
in the framework of Lipschitz differentiability spaces. 
We show that any tangent space of a Lipschitz differentiability space contains at least $n$ distinct tangent lines, 
obtained as the blow-up of $n$ Lipschitz curves, where $n$ is the dimension of the local measurable chart. 
Under additional assumptions on the space, such as curvature lower bounds, these $n$ distinct tangent lines span an
$n$-dimensional part of the tangent space.
\end{abstract}

\maketitle
%\tableofcontents

\section{Introduction}

During the past few years there has been growing interest towards studying the infinitesimal structure of 
``nice'' metric measure spaces. One class of nice metric measure spaces is formed by the ones in which
Lipschitz functions are differentiable almost everywhere with respect to Lipschitz charts covering the space.
The study of such spaces originates from the work of Cheeger \cite{cheeger:lip} and the spaces are now often called
Lipschitz differentiability spaces (following Bate \cite{bate:measurelip}).
Cheeger proved that a doubling condition on the reference measure and the validity of a local Poincar\'e inequality 
(as defined by Heinonen and Koskela \cite{HK}) are sufficient for the space to be a Lipschitz differentiability space. 
Although there are quite wild examples of doubling metric measure spaces supporting a local Poincar\'e inequality
\cite{BB,laakso,semmes2}, these assumptions still have strong geometric implications, \cite{cheeger:lip,semmes,korte}.
In particular, there are lots of rectifiable curves joining any two points in such a space.

A general Lipschitz differentiability space might not contain any rectifiable curves besides the trivial ones.
However, they always contain sufficiently many broken curves in different directions so that
the reference measure can be expressed by independent Alberti representations that completely characterize
derivatives of Lipschitz functions, see the work of Bate \cite{bate:measurelip}.
On the other hand, when we perform a Gromov-Hausdorff blow-up of a broken biLipschitz curve
$\gamma \colon \dom(\gamma) \to X$ of the metric space $X$ at a density point
of the domain $\dom(\gamma)$ the broken curve approaches, after passing to a subsequence,
a limit curve defined on the whole $\R$.

We first define metric differentiability, see Definition \ref{D:metricdiffer}, and then
prove that, at points of metric differentiability, 
this limit curve is a line-segment, see Proposition \ref{P:generaline}. 
By a result of Kirchheim \cite{kirch:diff}, we observe that metric differentiability coincides
 with the metric speed at almost every point of $\dom(\gamma)$.
Therefore we deduce that a Lipschitz curve $\gamma$ is metrically differentiable at almost every point
(see also Proposition \ref{P:metricdiff} for an alternative proof of this fact). 
Thus broken biLipschitz curves always converge to a line-segment at almost every point of their domain.

Given an $n$-dimensional Lipschitz chart on a Lipschitz differentiability space we know from the
work of Bate \cite{bate:measurelip} that there exist $n$ independent Alberti representations. 
Noticing that metric differentiability is Borel measurable, see Lemma \ref{L:Borel}, 
one can use it in the context of Lipschitz differentiability spaces to deduce that (see Proposition \ref{P:result1}) at almost every 
point the blow-up will give $n$ distinct tangent lines. 
If one also assumes the Lipschitz differentiability space to be doubling, then one can find $n$ distinct tangent
lines at almost every point of the tangent space.

\begin{theorem}[Theorem \ref{T:curves}]\label{T:1}
Let $(X,\sfd,\mm)$ be a doubling Lipschitz differentiability space and $(U,\f)$ be an $n$-dimensional chart. 
Then for $\mm$-almost every $\bar x \in U$, there exist $v_{1}, \dots, v_{n} \in \R^{n}$ linearly independent such that 
for any element $(X_{\infty},\sfd_{\infty},\bar x_{\infty}) \in \Tan(X, \sfd,\bar x)$  and 
for each $z \in X_{\infty}$ there exist 
$\iota^{z}_{1}, \dots, \iota^{z}_{n} \colon \R \to X_{\infty}$ so that \medskip
\begin{itemize}
	\item[i)] 	$\iota^{z}_{j}(0) = z$, for any $j = 1,\dots, n$; \\
	\item[ii)] 	$\sfd_{\infty}(\iota^{z}_{j}(t),\iota^{z}_{j}(s)) = |t - s|$, for any $j = 1,\dots, n$, for all $s,t \in \R$; \\
	\item[iii)] 	$\sfd_{\infty}(\iota^{z}_{j}(t),\iota^{z}_{k}(t)) \geq C |t| \cdot |v_{j} - v_{k}|$, for any $j, k = 1,\dots, n$, for all $t \in \R$; \medskip
\end{itemize}
for some positive constant $C$. For each $z \in X_{\infty}$, each line $\iota^{z}_{i}$ is obtained as the blow-up of a Lipschitz curve, with the blow-up depending on $z$.
\end{theorem}

The question is then how and what kind of subspace of the tangent space these tangent lines form.
Since the Heisenberg group is a Lipschitz differentiability space and
purely 2-unrectifiable \cite{AK2000}, we know that the tangent lines do not always
span an $n$-rectifiable set. However under the additional assumption
that the space is Ahlfors $n$-regular with $n$ being the dimension of the chart, at almost every point there is a
tangent space biLipschitz equivalent to $\R^n$, see \cite{david}.
We are interested in finding other conditions that would provide information on the tangents.
\medskip

Our considerations originate from the study of another class of nice metric measure spaces - namely of those
with Ricci curvature lower bounds. There are many notions of Ricci curvature
lower bounds on metric measure spaces.
For the most strict one, the $\mathsf{RCD}^{*}(K,N)$ spaces (defined in \cite{AGS11b,AmbrosioGigliMondinoRajala12,EKS,AMS}),
it is known that they infinitesimally look like Euclidean spaces, \cite{GMR,mondino:tangent}.
Moreover, the tangents in an $\mathsf{RCD}^{*}(K,N)$ space are almost everywhere spanned by the tangent lines obtained from the
Lipschitz charts as described above, see Section \ref{sec:split} for details.
Thus the infinitesimal structure of $\mathsf{RCD}^{*}(K,N)$ spaces is already well understood.

We would like to understand the structure of spaces with Ricci curvature lower bounds with the more general definitions.
Most of the definitions are known to imply doubling condition 
on the measure and a local Poincar\'e inequality. Thus these spaces are Lipschitz differentiability spaces and Theorem \ref{T:1} holds.
One line of investigation is to continue from the proof in \cite{GMR}. There the fact that $\mathsf{RCD}^{*}(K,N)$ spaces
have at least one Euclidean tangent space was proven
following the idea of Preiss \cite{P1987} (and its adaptation to metric spaces by Le Donne \cite{LD2011})
of iterated tangents. The proof essentially used only the fact that the tangent spaces split off any part that is isometric to $\R$.

Taking into consideration also the Lipschitz charts, the isometric splitting property implies the existence of $\R^n$
in each of the tangent at almost every point, where $n$ is again the dimension of the chart.

\begin{theorem}[Theorem \ref{T:n=k}]\label{T:2}
Suppose that $(X,\sfd,\mm)$ is a doubling Lipschitz differentiability space with the splitting of tangents property.
Let $(U,\f)$ be an $n$-dimensional chart of $(X,\sfd,\mm)$. 
Then for $\mm$-a.e. $\bar x \in U$  any $(X_{\infty}, \sfd_{\infty},\bar x_{\infty}) \in \Tan(X,\sfd,\bar x)$ is of the form
$$
(X_{\infty}^{d}\times \R^{d},\sfd^{d}_{\infty}\times | \cdot |_,(\bar x^{d}_{\infty}, 0)),
$$
with $d\geq n$. 
\end{theorem}

For the more general $\mathsf{CD}(K,N)$ spaces (see \cite{sturm:I,sturm:II,lottvillani:metric} for the definitions)
isometric splitting is impossible since already $\R^n$ with any norm and the Lebesgue measure satisfies $\mathsf{CD}(0,n)$.
On the other hand, Ohta has recently shown that a version of splitting theorem holds for Finsler manifolds \cite{Ohta}.
Such weaker versions might be enough to give some information on the infinitesimal structure. For example, if the
existence of a tangent line would always imply that the tangent could be written to be biLipschitz equivalent to a product
$\R \times Y$ for some metric space $Y$, the $n$ dimensional Lipschitz chart could result in a piece of the tangent biLipschitz
equivalent to $\R^n$.

Let us note that for the even more general notion $\mathsf{MCP}(K,N)$ of Ricci curvature lower bound (see \cite{sturm:II,Ohta2007}
for the definitions) the above splitting result does not hold even in a topological sense \cite{KR}.
Moreover, it is not known if a local Poincar\'e inequality holds in $\mathsf{MCP}(K,N)$ spaces
without the non-branching assumption, and hence we do not know if $\mathsf{MCP}(K,N)$ spaces are Lipschitz differentiability
spaces. Even more, it is known that for example the Heisenberg group satisfies the $\mathsf{MCP}(K,N)$ condition, see \cite{Juillet2009}.
Thus the tangent lines cannot biLipschitz span a part of the tangent.

\medskip
The paper is organized as follows. In Section \ref{sec:preli} we recall the notions of pointed measured Gromov-Hausdorff
convergence, tangent functions and Lipschitz differentiability spaces. In Section \ref{S:tangentline}
we define the notion of metric differentiability that we will use in this paper and show, using an identity proved by Kirchheim in \cite{kirch:diff}, 
that it agrees almost everywhere with the metric speed. We also show that at almost every point of a biLipschitz curve the
blow-up will be a tangent line. 
In Section \ref{sec:tanlipspace} we consider the blow-ups in Lipschitz differentiability space
showing that we have $n$ independent tangent lines at almost every point. In the final section, Section \ref{sec:split},
following the ideas of David and Schioppa \cite{david,schioppa}, we prove that if tangents split off tangent lines then the $n$ independent tangent lines in a Lipschitz differentiability
space span a Euclidean $\R^n$ in the tangent.
%{\color{red}Continue...}

\medskip

After the completion of this note, we learnt from Schioppa and Preiss that our approach can be used together 
with the very recent work by Cheeger, Kleiner and Schioppa \cite{CKS} to improve Theorem \ref{T:2} and show that at almost every point 
$X^{d}_{\infty} \times \R^{d} = \R^{n}$, where $n$ is the dimension of the chart.

\bigskip

%Consider a metric space $(X,\sfd)$ and $g \colon \R^n \to (X,\sfd)$. Define
%%
%$$
%MD(g,x)(u) := \lim_{r \searrow 0}\frac{1}{r}\sfd( g(x+ru), g(x))
%$$
%%
%for all $x,u \in \R^n$, whenever the limit exists. 
%The next result proved in \cite{kirch:diff}, see Theorem 2, shows that for Lipschitz functions $MD$ exists almost every where, with respect to Lebesgue measure, 
%and at almost every point where it exists, it is a semi-norm.
%
%\begin{theorem}[Kirchheim]\label{T:Kirch}
% Let $g \colon \R^n \to X$ be Lipschitz. Then, for almost every $x \in \R^n$, $MD(g,x)(\cdot)$ is a seminorm on $\R^n$ and
%%
%\begin{equation}\label{E:Kirch}
%  \sfd(g(z),g(y)) - MD(g,x)(z-y) = o(|z-x|+|z-y|).
%\end{equation}
%% 
%\end{theorem}
%
%In the case of Lipschitz curves, that is $n=1$, it is fairly easy to observe that at any point where \eqref{E:Kirch} holds true, 
%any tanget space of $(X,\sfd)$ at that point will contain an isometric copy of $\R$, a tangent line, constructed as the blow up of the curve.
%
%Nevertheless Theorem \ref{T:Kirch} does not give a criterion to check \eqref{E:Kirch} at the given point. 
%the tangent line 
%In this note, at least for the case $n=1$, 
%
%

%%%%%%%%%%%%%%%%%%%%%%%%%%%%%%%%%%%%%%%%%%%%%%%%%%%%%%
%%%%%%%%%%%%%%%%%%%%%%%%%%%%%%%%%%%%%%%%%%%%%%%%%%%%%%
%%%%%%%%%%%%%%%%%%%%%%%%%%%%%%%%%%%%%%%%%%%%%%%%%%%%%%
%%%%%%%%%%%%%%%%%%%%%%%%%%%%%%%%%%%%%%%%%%%%%%%%%%%%%%
%%%%%%%%%%%%%%%%%%%%%%%%%%%%%%%%%%%%%%%%%%%%%%%%%%%%%%
%%%%%%%%%%%%%%%%%%%%%%%%%%%%%%%%%%%%%%%%%%%%%%%%%%%%%%
%---------------PRELIMINARIES-----------------------------------------------

\section{Preliminaries}\label{sec:preli}
A metric measure space is a triple $(X,\sfd,\mm)$ where $(X,\sfd)$ is a complete and separable metric space and
$\mm$ a positive Borel measure that is also finite on bounded sets.
As the main object of our study will be proper spaces, i.e. metric spaces such that each bounded closed set is also compact, we directly incorporate 
in the definition of metric measure space also the properness assumption. Consequently $\mm$ will be a positive Radon measure.

We list here two general properties of metric measure space that we will consider during the paper.  
The metric measure space $(X,\sfd,\mm)$ is \emph{doubling} if for each $R>0$ there exists $C(R)>0$ such that
$$
0 < \mm(B_{2r}(x)) \leq C(R) \,\mm(B_{r}(x)), \qquad \text{for every } x\in X, \, r\leq R.
$$
With no loss in generality, the function $C$ can be taken non-decreasing. 
Morevover a metric measure space $(X,\sfd,\mm)$ supports a \emph{local $p$-Poincar\'e inequality} for some $p\geq1$ 
if every ball in $X$ has positive and finite measure and for every $g \in \Lip(X,\sfd) := \{ l \colon X \to \R | \, l \textrm{ is Lipschitz} \}$,
$$
\fint_{B}| g(x) - g_{B}|\,d\mm(x) \leq L r \left( \fint_{B_{rL}(x_{0})} |D g|^{p}(x)\,d\mm(x) \right)^{1/p},
$$
for some positive constant $L$, where $B = B_{r}(x_0)$ and $g_{B} = \fint_{B}g(x) \,d\mm(x)$. Here for $g \in \Lip(X,\sfd)$ we also adopt the following notation: 
$$
|D g|(x) : = \sup_{y \neq x} \frac{\sfd( g(y), g(x) ) }{\sfd(y,x)}.
$$
%

%%%%%%%%%%%%%%%%%%%%%%%%%%%%%%%%%%%%%%%%%%%%%%%
%%%--------------CONVERGENCE OF PMMS----------------------------------

\subsection{Convergence of metric measure spaces}

The standard notion of topology on equivalence classes of pointed, proper, separable metric spaces is the one induced by the
pointed Gromov-Hausdorff convergence, $pGH$-convergence in brief. This convergence can be characterized in many equivalent ways. 
We will adopt the one with $\ve$-isometries.

%Assume for the moment the existence of a pointed, locally compact, complete and separable metric space $(X_{\infty}, d_{\infty},\bar x_{\infty})$. 
%Assume also that 
%\[
%\left(X,\frac{1}{r_{i}} d, \bar x \right) \to (X_{\infty}, d_{\infty},\bar x_{\infty}), \qquad \textrm{pointed Gromov-Hausdorff},
%\]
%for some sequence $r_{i} \to 0$. 

A map $f \colon (X,\sfd_{X}) \to (Y,\sfd_{Y})$ between compact metric spaces is called an \emph{$\ve$-isometry} provided
\begin{itemize}
\item[(i)] it almost preserves distances: for all $z,w \in X$,
$$
| \sfd_{X}(z,w) - \sfd_{Y}(f(z),f(w))| \leq \ve;
$$
\item[(ii)] it is almost surjective: 
$$
\forall \ y \in Y, \ \  \exists \ x \in X \ : \quad \sfd_{Y}(f(x),y) \leq \ve.
$$
\end{itemize}
In order to deal with possibly non-compact spaces, it is customary to fix a distinguished point $\bar x \in X$ and to consider $\ve$-isometries defined 
on an increasing family of balls centered in $\bar x$. When a distinguished point is fixed, we use $(X,\sfd,\bar x)$ to denote the pointed metric space.

\begin{definition}\label{D:pGH}
A sequence $\{(X_i,\sfd_i,\bar x_i)\}_{i\in \N}$ of pointed, proper, complete metric spaces converges to a pointed, proper, 
complete metric space $(X_\infty,\sfd_\infty,\bar x_\infty)$,
$$
(X_i,\sfd_i,\bar x_i) \longrightarrow (X_{\infty}, \sfd_{\infty},\bar x_{\infty}), \qquad pGH
$$
if and only if there exist sequences of positive real numbers $\{\ve_{i}\}_{i \in \N}, \{R_{i}\}_{i \in \N}$ with $\ve_{i} \to 0$, $R_{i} \to \infty$ and a sequence of 
$\ve_{i}$-isometries, 
$$
f_{i} \colon B^{X_{i}}_{R_{i}}(\bar x_{i}) \longrightarrow B^{X_{\infty}}_{R_{i}}(\bar x_{\infty}), \qquad f_{i}(\bar x_{i}) = \bar x_{\infty}, 
$$
where $B^{X_{i}}_{R_{i}}(\bar x_{i})$ is the ball in $X_i$, centered in $\bar x$ and of radius $R_{i}$.
\end{definition}

%We will use the notation $\sfd_{pGH}$ for the $pGH$-convergence.
We also consider pointed metric measure spaces: a quadruple $(X,\sfd,\mm,\bar x)$ where $(X,\sfd,\mm)$ is a metric measure space
and $\bar x \in X$ a distinguished point.

\begin{definition}\label{D:pmGH}
A sequence $\{(X_{i},\sfd_{i},\mm_{i},\bar x_{i})\}_{i\in \N}$ of pointed metric measure spaces converges in the pointed measured Gromov-Hausdorff convergence
to a pointed metric measure space $(X_{\infty},\sfd_{\infty},\mm_{\infty},\bar x_{\infty})$ 
$$
(X_i,\sfd_i,\mm_{i},\bar x_i) \longrightarrow (X_{\infty}, \sfd_{\infty},\mm_{\infty},\bar x_{\infty}), \qquad pmGH
$$
if and only if there exist sequences of positive real numbers $\{\ve_{i}\}_{i \in \N}, \{R_{i}\}_{i \in \N}$
with $\ve_{i} \to 0$, $R_{i} \to \infty$ and a sequence of $\ve_{i}$-isometries, 
$$
f_{i} \colon B^{X_{i}}_{R_{i}}(\bar x_{i}) \longrightarrow B^{X_{\infty}}_{R_{i}}(\bar x_{\infty}), \qquad f_{i}(\bar x_{i}) = \bar x_{\infty}, 
$$
such that  
$$
\lim_{i \to \infty} \int_{X_{\infty}} \f(z) \,d(f_{i \,\sharp}\mm_{i})(z) = \int_{X_{\infty}} \f(z)\,d\mm_{\infty}(z), \qquad \forall \, \f \in C_{b}(X_{\infty}), 
$$
where $C_{b}(X_{\infty})$ stands for the space of continuous and bounded functions with compact support in $X_{\infty}$.
\end{definition}

Both, the $pGH$-convergence and the $pmGH$-convergence can be used to define and study (measured) tangent spaces.

If $(X,\sfd)$ is a metric space and $\bar x \in X$ is a distinguished point, 
then any limit point in the $pGH$-convergence of any sequence of the form $\{ (X,\sfd/r_{i},\bar x) \}_{i \in \N}$, with $r_{i} \to 0$, 
is a tangent space of $(X,\sfd)$ at $\bar x$. 
We use $\Tan(X,\sfd,\bar x)$ to denote the set of all possible tangent spaces of $(X,\sfd)$ at $\bar x$.

If $(X,\sfd,\mm)$ is a metric measure space and $\bar x \in \supp(\mm)$ is a distinguished point, for any $r>0$, the rescaled and normalized 
pointed metric measure space is defined as follows:
$$
\left( X, \frac{1}{r} \sfd,\mm^{\bar x}_{r},\bar x \right), \qquad \mm^{\bar x}_{r} : = \left( \int_{B_{r}(\bar x)} 1 - \frac{1}{r}\sfd(\bar x,z) \,d\mm(z)\right)^{-1} \mm.
$$
Then a limit point in the $pmGH$-convergence of the sequence $\{ (X,\sfd/r_{i},\mm^{\bar x}_{r_{i}},\bar x) \}_{i \in \N}$ 
is a measured tangent space of $(X,\sfd,\mm)$ at $\bar x$ and 
to denote the set of all possible measured tangent spaces of $(X,\sfd,\mm)$ at $\bar x$ we use $\Tan(X,\sfd,\mm,\bar x)$.

It is worth noticing that, thanks to compactness properties of the collection of uniformly doubling metric measure space, 
$\Tan(X,\sfd,\mm,\bar x)$ is always non empty, provided $(X,\sfd,\mm)$ is doubling.

%%%%%%%%%%%%%%%%%%%%%%%%%%%%%%%%%%%%%%%%%%%%%%%
%%%-----MOVING THE BASE POINT-------------------%%%%%%%%%%%%%%%%%%%%%

%%%%%%%%%%%%%%%%%%%%%%%%%%%%%%%%%%%%%%%%%%%%%%%
%%%--------------TANGENT FUNCTIONS----------------------------------%%%%%%%%%%%%%%

\subsection{Tangent functions}\label{Ss:tangetfunct} 
Here we recall few objects and related results presented in \cite{cheeger:lip} and in \cite{keith:diff}.

%We have used $\ve$-isometries to study measured tangent spaces, while for introducing tangent functions
%we will use recall the existence of approximate inverse.
If $(X, \sfd_{X})$ and $(Y, \sfd_{Y})$ are metric spaces and $f \colon X \to Y$ is an $\ve$-isometry, 
then there exists a $(4\ve)$-isometry $f' \colon Y \to X$ so that for all 
$x \in X$ and $y \in Y$ it holds 
$$
\sfd_{X}( f' \circ f (x) , x ) \leq 3 \ve, \qquad  \sfd_{Y}( f \circ f' (y) , y ) \leq \ve.
$$
Such a map is usually called an \emph{$\ve$-inverse} of $f$ and accordingly we will often adopt the notation $f^{-1}$ to denote it. 

Consider now any element $(X_{\infty},\sfd_{\infty},\mm_{\infty},\bar x_{\infty}) \in \Tan(X,\sfd,\mm,\bar x)$ and a sequence of $r_{i} \to 0$ such that 
$$
\left(X,\frac{1}{r_{i}}\sfd,\mm^{\bar x}_{r},\bar x\right) \longrightarrow (X_{\infty},\sfd_{\infty},\mm_{\infty},\bar x_{\infty}),\qquad  pmGH.
$$
Then to any Lipschitz function $g \colon X \to \R$ we can associate a sequence of rescaled functions centered in $\bar x$: 
$$
g_{i}(x) : = \frac{g(x) - g(\bar x)}{r_{i}}. 
$$
If $g$ is $L$-Lipschitz in $(X,\sfd)$, then so is $g_{i}$ in $(X, \sfd/r_{i})$. 
With this in mind, we say that $u_{g} \colon X_{\infty} \to \R$ is a \emph{compatible tangent function of $g$ at $x$} if
$$
\lim_{i \to \infty} g_{i}(f^{-1}_{i}(z)) = \lim_{i \to \infty} \frac{g(f^{-1}_{i} (z) )   - g(\bar x)  }{r_{i}} = u_{g}(z), \qquad \forall \, z \in X_{\infty},
$$
where $f^{-1}_{i}$ is any $\ve_{i}$-inverse of the approximate isometry $f_{i}$ given by the $pmGH$ convergence of 
$(X, \sfd/r_{i},\mm^{\bar x}_{r},\bar x)$ to $(X_{\infty}, \sfd_{\infty},\mm_{\infty},\bar x_{\infty})$. 
The term \emph{compatible} is used to underline that we used the same scaling for the distance and the function $g$.

\begin{remark}
The definition of $u_{g}$ does not depend on the choice of the sequence of the $\ve_{i}$-inverses.
Since $f_{i}$ is almost surjective, for any $z \in X_{\infty}$ and $i \in \N$ sufficiently large, there exists $x_{i} \in X$ such that 
$$
\sfd_{\infty} (f_{i}(x_{i}) , z  ) \leq \ve_{i}.
$$
One then easily observes that $|g_{i}(f^{-1}_{i}(z))  - g_{i}(f^{-1}_{i} \circ f_{i}(x_{i})) | \to 0$. 
If $f_{i}^{-1}$ and $\hat f_{i}^{-1}$ are two distinct $\ve_{i}$-inverses of 
$f_{i}$, it follows, by the triangle inequality that 
$$
\lim_{i \to \infty} \frac{1}{r_{i}} \sfd (  f^{-1}_{i} \circ f_{i}(x_{i}), \hat f^{-1}_{i} \circ f_{i}(x_{i}) ) = 0,
$$
and since $g$ is Lipschitz, it follows that $g_{i}(f^{-1}_{i}(z))$ and $g_{i}(\hat f^{-1}_{i}(z))$ have the same limit.  
\end{remark}

\medskip

Concerning the existence of compatible tangent functions, the following compactness result holds. 
\begin{lemma}\label{L:tgfunct}
Let $(X,\sfd,\mm)$ be a doubling metric measure space and a sequence $r_{i} \to 0$ such that 
$$
\left(X,\frac{1}{r_{i}}\sfd,\mm^{\bar x}_{r},\bar x \right) \longrightarrow (X_{\infty},\sfd_{\infty},\mm_{\infty},\bar x_{\infty}) \in \Tan(X,\sfd,\mm,\bar x),
$$
where the convergence is in the $pmGH$ sense. 
Fix also a countable collection $\mathcal{F}$ of uniformly Lipschitz functions defined on $X$.
Then possibly choosing a subsequence of $\{ r_{i} \}_{i \in \N}$, for each $g \in \mathcal{F}$ there exists $u_{g}$ 
compatible tangent function of $g$ at $\bar x$.
\end{lemma}

As one might expect, tangent functions of Lipschitz functions enjoy a generalized notion of linearity. It has different names according to different authors. 
Here we follow \cite{cheeger:lip} and say that tangent functions to Lipschitz functions, 
wherever they exists, are \emph{generalized linear}, see Definition 8.1 of \cite{cheeger:lip}. 
The terminology used is justified by the fact that being generalized linear on a Euclidean space is the same as being linear in the usual sense, 
see again \cite{cheeger:lip}, Theorem 8.11.

%%%%%%%%%%%%%%%%%%%%%%%%%%%%%%%%%%%%%%%%%%%%%%%
%%%--------------DIFFERENTIABILITY SPACES----------------------------------

\subsection{Lipschitz differentiability spaces}\label{Ss:diffspaces}

%̃\\
%{\color{blue}From the paper by Kirchheim:}

Under fairly general assumption on the structure of the metric measure space, it is proved in \cite{cheeger:lip} that the space of Lipschitz functions has finite dimension in the following sense. 
%We report here few definitions taken from  and 

\begin{definition}
Let $(X,\sfd)$ be a metric space and $n \in \N$. A Borel set $U \subset X$ and a Lipschitz function $\f \colon X \to \R^{n}$ form a \emph{chart of dimension n}, $(U,\f)$, and 
a function $g \colon X \to \R$ is \emph{differentiable at $x_{0} \in U$} \emph{with respect to} $(U,\f)$ if there exists 
a unique $Dg(x_{0}) \in \R^{n}$ such that 
$$
\limsup_{x\to x_{0}} \frac{| g(x) - g(x_{0})   - Dg(x_{0})\cdot(\f(x) - \f(x_{0})) |}{\sfd(x,x_{0})} = 0.
$$
Furthermore a metric measure space $(X,\sfd,\mm)$ is called a \emph{Lipschitz differentiability space} if there exists a countable decomposition of $X$ into charts such that any Lipschitz function $g \colon X \to \R$ is differentiable at $\mm$-almost every point of every chart.
\end{definition}

A celebrated result by Cheeger \cite{cheeger:lip} on Lipschitz differentiability spaces can be summarized by the following
\begin{theorem}\label{T:cheeger}
Let $(X,\sfd,\mm)$ be a doubling metric measure space supporting a $p$-Poincar\'e inequality 
with constant $L\geq1$ for some $p\geq 1$. Then $(X, \sfd,\mm)$ is a Lipschitz differentiability space. 
\end{theorem}

Subsequently in \cite{bate:measurelip} a finer analysis on curves, and their possible directions with respect to a given chart, was carried on. 
Here we report only the main statement. 
We use $\Gamma(X)$ to denote the set of biLipschitz maps
$$
\gamma \colon \dom(\gamma) \to X,
$$
with $\dom(\gamma) \subset \R$ non-empty and compact.

\begin{theorem}[\cite{bate:measurelip}, Theorem 6.6, Corollary 6.7]\label{T:bate}
Let $(X,\sfd,\mm)$ be a Lipschitz differentiability space and $(U,\f)$ an $n$-dimensional chart. 
Then for $\mm$-a.e. $x \in U$, there exist $\gamma_{1}^{x},\dots, \gamma_{n}^{x} \in \Gamma(X)$
such that each 
\begin{itemize}
\item[i)] $(\gamma_{i}^{x})^{-1}(x) = 0$ is a point of density one of $(\gamma_{i}^{x})^{-1}(U)$;  
\item[ii)] $(\f \circ \gamma_{i}^{x})'(0)$ are linearly independent.
\end{itemize}
Moreover, for any such $\gamma_{i}^{x}$, for any Lipschitz $g \colon X \to \R$ and $\mm$-a.e. $x \in U$, the gradient of $g$ at $x$ with respect to $\f$ 
and $\gamma_{1}^{x},\dots, \gamma_{n}^{x}$ equals $Dg(x)$, 
that is 
\[
\left(g \circ \gamma_{i}^{x}\right)'(0) = Dg(x)\cdot \left(\f\circ \gamma_{i}^{x}\right)'(0), \quad \mm-a.e.\, x \in U,
\]
for $i =1,\dots, n$.
\end{theorem}

Hence not only the space of Lipschitz functions has locally finite dimension but also each Lipschitz function 
is locally described in terms of directional derivative with respect to 
a family of biLipschitz curves.

Here also Keith's results on coordinate functions is worth mentioning: in \cite{keith:diff} it is proved that 
the role of coordinate map $\f$ in chart $(U,\f)$ can be played by distance functions from a well prepared set. We report here Theorem 2.7 of \cite{keith:diff}. 

\begin{theorem}
Let $(X,\sfd,\mm)$ be a complete and separable metric measure space admitting a $p$-Poincar\'e inequality with $\mm$ doubling. 
Then there exists a measurable differentiable structure $\{(U_{i},\f_{i})\}_{i\in \N}$ such that each $\f_{i} : U_{i} \to \R^{d(i)}$ is of the form 
$$
\f_{i}(z) = \left( \sfd(z,x_{1}), \dots, \sfd(z,x_{d(i)}) \right),
$$
for some $x_{1},\dots, x_{d(i)} \in X$,
\end{theorem}

%%%%%%%%%%%%%%%%%%%%%%%%%%%%%%%%%%%%%%%%%%%%%%%%%%%%%%
%-----------GEODESICS IN PRODUCT SPACES------------------------------------

\subsection{Geodesics in product spaces}

If $(X,\sfd_{X})$ and $(Y, \sfd_{Y})$ are two metric spaces, we can consider the product distance $\sfd_{XY}$ defined by
$$
\sfd_{XY} : = \sqrt{ \sfd_{X}^{2} + \sfd_{Y}^{2} }.
$$
Then $(X\times Y, \sfd_{XY})$ is again a metric space. We recall an easy lemma on geodesics in product spaces.

\begin{lemma}\label{L:geodesicproduct}
A curve $[0,1] \ni t \mapsto (\gamma_{t}^{1}, \gamma_{t}^{2}) \in (X\times Y,\sfd_{XY})$ is a geodesic 
if and only if $\gamma^{1}$ is a geodesic in $(X,\sfd_{X})$ and $\gamma^{2}$ is a geodesic in $(Y,\sfd_{Y})$.
\end{lemma}

\begin{proof}
We start with the easy inequality: for $a,b,c,d$ positive real numbers, 
\begin{equation}\label{E:geodprod}
(a^{2}+ b^{2})(c^{2} + d^{2}) \geq (bd + ac)^{2}.
\end{equation}
Then let $[0,1] \ni t \mapsto (\gamma_{t}^{1}, \gamma_{t}^{2}) \in X\times Y$ be a geodesic and suppose by contradiction that $\gamma^{1}$ is not. 
For ease of notation, we can assume that
$$
\sfd_{X} (\gamma^{1}_{-s}, \gamma^{1}_{s}) < \sfd_{X} (\gamma^{1}_{0}, \gamma^{1}_{s}) + \sfd_{X} (\gamma^{1}_{0}, \gamma^{1}_{s}),
$$
for some $s> 0$.
On the other hand 
$$
\sfd_{X}^{2} (\gamma^{1}_{-s}, \gamma^{1}_{s}) + \sfd_{Y}^{2} (\gamma^{2}_{-s}, \gamma^{2}_{s}) 
	=  \left(  \sqrt{\sfd_{X}^{2} (\gamma^{1}_{-s}, \gamma^{1}_{0}) + \sfd_{Y}^{2} (\gamma^{2}_{-s}, \gamma^{2}_{0})} 
		+  \sqrt{\sfd_{X}^{2} (\gamma^{1}_{0}, \gamma^{1}_{s}) + \sfd_{Y}^{2} (\gamma^{2}_{0}, \gamma^{2}_{s})} \right)^{2}.
$$
Expanding the squares and using that $\gamma^{1}$ is not a geodesic, we obtain that
\begin{align*}
\sfd_{Y}^{2} (\gamma^{2}_{-s}, \gamma^{2}_{s}) 
	>	&~ \sfd_{Y}^{2} (\gamma^{2}_{-s}, \gamma^{2}_{0}) + \sfd_{Y}^{2} (\gamma^{2}_{0}, \gamma^{2}_{s})  \crcr
		&~ +2  \sqrt{\sfd_{X}^{2} (\gamma^{1}_{-s}, \gamma^{1}_{0}) + \sfd_{Y}^{2} (\gamma^{2}_{-s}, \gamma^{2}_{0})} \cdot
				\sqrt{\sfd_{X}^{2} (\gamma^{1}_{0}, \gamma^{1}_{s}) + \sfd_{Y}^{2} (\gamma^{2}_{0}, \gamma^{2}_{s})} \crcr
		&~ - 2 \sfd_{X}(\gamma^{1}_{-s}, \gamma^{1}_{0}) \sfd_{X}(\gamma^{1}_{0}, \gamma^{1}_{s}).
\end{align*}
We can now use the first inequality we wrote and get 
$$
\sfd_{Y}^{2} (\gamma^{2}_{-s}, \gamma^{2}_{s}) 
		>  \sfd_{Y}^{2} (\gamma^{2}_{-s}, \gamma^{2}_{0}) + \sfd_{Y}^{2} (\gamma^{2}_{0}, \gamma^{2}_{s})  
			+ 2 \sfd_{Y}(\gamma^{2}_{-s}, \gamma^{2}_{0}) \sfd_{Y}(\gamma^{2}_{0}, \gamma^{2}_{s}),
$$
violating the triangle inequality. The claim follows.
\end{proof}

%%%%%%%%%%%%%%%%%%%%%%%%%%%%%%%%%%%%%%%%%%%%%%%%%%%%
%%%%%%%%%%%%%%%%%%%%%%%%%%%%%%%%%%%%%%%%%%%%%%%%%%%%
%%%%%%%%%%%%%%%%%%%%%%%%%%%%%%%%%%%%%%%%%%%%%%%%%%%%
%%%%%%%%%%%%%%%%%%%%%%%%%%%%%%%%%%%%%%%%%%%%%%%%%%%%
%%%%%%%%%%%%%%%%%%%%%%%%%%%%%%%%%%%%%%%%%%%%%%%%%%%%
%-------------------TANGENT LINES-----------------------------------%

\section{Tangent lines}\label{S:tangentline}

Let us start this section by recalling a result from \cite{rudin:real}, Theorem 7.10:
a more general version of Lebesgue Differentiation Theorem.
Here and in the sequel $\L^{d}$ denotes the Lebesgue measure on $\R^{d}$.

\begin{definition} 
Fix $x \in \R^{d}$ and a sequence of Borel sets $\{ E_{i}\}_{i\in \N} \subset \R^{d}$. 
We say that $\{E_{i} \}_{i\in \N}$ \emph{shrinks nicely to x} provided there exist $r_{i} >0$ and $\alpha >0$ such that 
for each $i \in \N$ we have
$$
E_{i} \subset B_{r_{i}}(x) \qquad \text{and} \qquad
\L^{d}(E_{i}) \geq \alpha  \L^{d}(B_{r_{i}}(x)).
$$
\end{definition}

For the nicely shrinking sets we have the following general version of Lebesgue Differentiation Theorem.

\begin{theorem}\label{T:diffLeb} 
Let $f \in L^{1}(\R^{d},\R)$ be any function. Associate to each $x \in \R^{d}$ a sequence $\{ E_{i}(x)\}_{i \in \N}$ of sets nicely shrinking to $x$. 
Then 
$$
f(x) = \lim_{i \to \infty} \frac{1}{\L^{d}(E_{i}(x))} \int_{E_{i}(x)} f(y) dy,
$$
for every Lebesgue point $x$ of $f$. In particular it holds for $\L^{d}$-almost every $x$.
\end{theorem}

Consider now $(X,\sfd)$ a complete, and separable metric space and note that for the next statement we do not need to assume $(X,\sfd)$ to be proper. 
%
%Consider $\gamma \in \Gamma(X)$ so that: 
%\[
%\gamma \colon [-c, c] \to X, \qquad \gamma_{0} = \bar x,
%\]
%with $t = 0$ point of metric differentiability. Since $\gamma \in \Gamma(X)$, it holds $|d\gamma_{0}|>0$.
%
%
%The natural set of curves we want to work with is $\Gamma(X)$: the set of bi-Lipschitz curves
%%
%$$
%\gamma : \textrm{Dom\,}(\gamma) \to X,
%$$
%%
%with $\textrm{Dom\,}\gamma \subset \R$ non-empty and compact. 

\begin{definition}\label{D:metricdiffer}
Let $\gamma \colon \dom(\gamma)\to X$ be any curve. We say that $\gamma$ is \emph{metric differentiable} at $t \in \dom(\gamma)$ provided the following limit 
$$
\lim _{\substack{s ,\tau \to 0 \\ nicely}}  \frac{\sfd(\gamma_{t+s}, \gamma_{t+\tau})}{|s - \tau|}
$$
exists for any sequence of $s$ and $\tau$, where with \emph{nicely} we ask for the interval with boundary formed by $t+s$ and $t+\tau$ to shrink nicely to $t$. 
In case the limit exists, we denote it with $|d\gamma|(t)$.
\end{definition}

\begin{remark}\label{R:comparespeed}
By definition, the existence of $|d\gamma|$ is a priori a more demanding property compared to existence of metric speed $|\dot \gamma|$, 
for its definition see \cite{ambrtilli:metricanal}.
Actually the two notions are different. Consider for instance the curve $\gamma\colon [-1,1 ] \to \R^{2}$ defined by 
$\gamma(t) : = (t,t)$ for $t \geq 0$ and $\gamma(t) : = (t,-t)$ for $t \leq 0$. Then the metric speed always exists and is $1$, while $|d\gamma|$ does not exists for $t =0$.
The converse trivially holds.
For curves with values in a Euclidean space, at any point of differentiability, $|d\gamma|(t_{0})$ coincides with the modulus of the derivative. 
\end{remark}

\begin{remark}\label{R:kirchh}
Another notion of differentiability for maps with values in metric spaces was introduced by Kirchheim in \cite{kirch:diff}:
for any $g \colon \R^n \to (X,\sfd)$ consider the following quantity
$$
MD(g,x)(u) := \lim_{r \searrow 0}\frac{1}{r}\sfd( g(x+ru), g(x))
$$
for all $x,u \in \R^n$, whenever the limit exists. 
In Theorem 2 of \cite{kirch:diff} it is proved that for Lipschitz functions $g$, $MD$ exists almost everywhere, with respect to Lebesgue measure, 
and at almost every point where it exists, it is a seminorm.

\begin{theorem}[Kirchheim]\label{T:Kirch}
 Let $g \colon \R^n \to X$ be Lipschitz. Then, for almost every $x \in \R^n$, $MD(g,x)(\cdot)$ is a seminorm on $\R^n$ and
\begin{equation}\label{E:Kirch}
  \sfd(g(z),g(y)) - MD(g,x)(z-y) = o(|z-x|+|z-y|).
\end{equation}
\end{theorem}
In the case of Lipschitz curves ($n=1$) the quantity $MD$ coincides with the metric speed and at any point where it exists it is also a seminorm. 
As the objective of this paper is the study of tangent lines, \eqref{E:Kirch} is the relevant identity. 
It is straightforward to observe that if \eqref{E:Kirch} holds at $t \in \dom(\gamma)$ then $t$ is a point of metric differentiability and $|d\gamma|(t) = MD(\gamma,t)$.
Also the converse implication holds. We include here a short proof for reader's convenience.

\begin{lemma}\label{L:Kirchcomparison}

Suppose a Lipschitz curve $\gamma : [ - c, c] \to X$ is metric differentiable at $0$. 
Then 
$$
\sfd(\gamma_{t}, \gamma_{s}) - |d\gamma|(0) \cdot |t - s| = o(|t| + |t - s|).
$$
\end{lemma}
\begin{proof} 
Denote the Lipschitz constant of $\gamma$ by $L$. 
Let $\epsilon > 0$ . From the metric differentiability there exists $r_{\epsilon}>0$ such that if  $\epsilon |t| <|t - s|<|t|² r_{\epsilon}$,
we have
$$
\sfd (\gamma_{t}, \gamma_{s}) - |d\gamma|(0)\cdot |t - s| \, \leq \, \epsilon |t - s|\,  < \, \epsilon |t|.
$$
On the other hand, if $0 < |t - s| < \epsilon |t|$,
we have from the Lipschitz-continuity
$$
\sfd(\gamma_{t}, \gamma_{s}) - |d\gamma|(0)\cdot |t- s| \,  \leq \, 2L |t - s| \,  <  \, 2L \epsilon |t|.
$$
The claim follows by combining the estimates.
\end{proof}

%Since Theorem \ref{T:Kirch} does not give a criterion to check if \eqref{E:Kirch} holds true or not at a given point $x$ and 
%by construction metric differentiability is a more direct condition to check, 
In this paper we prefer to analyze the properties of $|d\gamma|$ rather than \eqref{E:Kirch}.

\end{remark}

Taking advantage of Theorem \ref{T:diffLeb}, it is fairly easy to obtain the almost everywhere existence of $|d\gamma|$.

\begin{proposition}\label{P:metricdiff}
Let $\gamma \colon \dom(\gamma) \to X$ be a Lipschitz curve. 
Then metric differentiability holds $\L^{1}$-a.e. in $\dom(\gamma)$.
\end{proposition}

The proof can be obtained already from what was said in Remark \ref{R:kirchh}. 
However, we present here an alternative proof obtained following the ideas of the proof of existence of the metric 
speed for $\L^{1}$-a.e. $t \in [0,1]$, see \cite{ambrtilli:metricanal}, Theorem 4.1.6.

\begin{proof}

{\bf Step 1.} \\
Consider $\Lambda : = \gamma(\dom(\gamma))$. 
By continuity of $\gamma$, the set $\Lambda$ is compact and we can consider a dense sequence $\{x_{n}\} \subset \Lambda$.
We define a sequence of Lipschitz function as follows: 
$$
\dom(\gamma) \ni t \mapsto \f_{n}(t) : = \sfd(\gamma_{t},x_{n}).
$$
The Lipschitz constant of $\f_{n}$ coincides with that of $\gamma$. For each $n \in \N$ we denote with $\hat \f_{n}$ a Lipschitz extension of $\f_{n}$.
We can assume $\hat \f_{n}$ to be defined on an interval, say on $(a,b)$, containing $\dom(\gamma)$. 
By Rademacher theorem, each $\hat \f_{n}$ is differentiable $\L^{1}$-a.e. and therefore we can define the following map
$$
p(t) : = \sup_{n\in \N} |\dot{\hat{\f}}_{n}(t)|,
$$
at least for almost every $t \in (a,b)$.
\medskip

{\bf Step 2.} \\
For the rest of the proof we fix $t \in \dom(\gamma)$ which is a point of differentiability of all $\f_{n}$ and a Lebesgue-point 
of $p$. 
We also fix two sequences $s_{m},\tau_{m} \to 0$ with $0 \leq s_{m} \leq \tau_{m}$ so that 
$$
t +s_{m}, t+ \tau_{m} \in \dom(\gamma), \qquad \frac{s_{m}}{\tau_{m}} \to \alpha \in [0,1).
$$
This last condition is equivalent to ask that the interval $(t+s_{m},t+\tau_{m})$ shrinks nicely to $t$. For ease of notation, $s = s_{m}, \tau = \tau_{m}$.
Then for any $n \in \N$ we have
$$
\frac{\sfd(\gamma_{t+s}, \gamma_{t+\tau})}{\tau - s} \geq \frac{ |\f_{n}(\gamma_{t+s})   - \f_{n}(\gamma_{t+\tau}) |  }{\tau - s}, 
$$
and therefore
$$
\liminf_{s,\tau \to 0} \frac{\sfd(\gamma_{t+s}, \gamma_{t+\tau})}{\tau - s} \geq |\dot \f_{n}(t)|.
$$
We can take the supremum over all $n$ without changing the left hand side of the previous inequality and obtaining on the right hand side $p(t)$.
\medskip

{\bf Step 3.} \\
Since $\{x_{n}\}_{n\in \N}$ is a dense sequence
\begin{align*}
\sfd(\gamma_{t+s},\gamma_{t+\tau}) = &~ \sup_{n\in \N} |\sfd(\gamma_{t+s},x_{n} ) - \sfd(x_{n}, \gamma_{t+\tau}) | \crcr
= &~  \sup_{n \in \N} |\f_{n}(t+s) -  \f_{n}(t+\tau) | \crcr
\leq &~ \sup_{n \in \N} \int_{s}^{\tau} |\dot{\hat{\f}}_{n}(\sigma)| \L^{1}(d\sigma) \crcr
\leq &~  \int_{t+s}^{t+\tau} p(\sigma) \L^{1}(d\sigma).
\end{align*}
By assumption $t$ is a Lebesgue-point of $p$, then
$$
\limsup_{s,\tau \to 0} \frac{\sfd(\gamma_{t+s},\gamma_{t+\tau})}{\tau -s} \leq \limsup_{s,\tau \to 0}  \frac{1}{\tau - s} \int_{t+s}^{t+\tau} p(\sigma) \L^{1}(d\sigma) = p(t),
$$
where the last identity follows from Theorem \ref{T:diffLeb}.
For $\mathcal{L}^{1}$-a.e. $t \in \dom(\gamma)$:
$$
p(t) \leq \liminf_{\substack{s ,\tau \to 0 \\ nicely}} \frac{\sfd(\gamma_{t+s},\gamma_{t+\tau})}{\tau -s} \leq 
\limsup_{\substack{s ,\tau \to 0 \\ nicely}} \frac{\sfd(\gamma_{t+s},\gamma_{t+\tau})}{\tau -s} \leq p(t), 
$$
and the claim follows.
\end{proof}

\medskip

%%%%%%%%%%%%%%%%%%%%%%%%%%%%%%%%%%%%%%%%%%%%%%%%%%%%%%
%----------------------EXISTENCE OF TANGENT LINES--------------------------------------

\subsection{Existence of Tangent lines}

From Proposition \ref{P:metricdiff} one can prove that at each point of metric differentiability the blow up of the Lipschitz curve is a tangent line.   
Note that we use the properness assumption of the base space $(X,\sfd)$ in the proof.
%Note that the properness assumption on the base space $(X,\sfd)$ is needed.

\begin{lemma}\label{L:line} 
Let $(X,\sfd)$ be a complete, proper and separable metric space. Fix $\bar x \in X$ and assume 
the existence of a pointed, proper, complete and separable metric space $(X_{\infty}, \sfd_{\infty},\bar x_{\infty}) \in \Tan(X,\sfd,\bar x)$.
Let $\gamma \in \Gamma(X)$ be such that
$$ 
\dom(\gamma) = [-c,c], \qquad  \gamma_{0} = \bar x, \qquad |d\gamma|(0)>0. 
$$
Then $(X_{\infty}, \sfd_{\infty},\bar x_{\infty})$ contains an isometric copy of $\R$, in brief a line, limit of $\gamma$.
\end{lemma}

\begin{proof}
By assumption there exists a sequence of positive real numbers $\{r_{i} \}_{i\in \N}$ with $r_{i} \to 0$ such that 
\[
\left(X,\frac{1}{r_{i}} \sfd, \bar x \right) \to (X_{\infty}, \sfd_{\infty},\bar x_{\infty}) \qquad\text{in the }pGH\text{-convergence}.
\]
%$$
%\sfd_{pGH}\left(\left(X,\frac{1}{r_{i}} \sfd, \bar x \right), (X_{\infty}, \sfd_{\infty},\bar x_{\infty}) \right) \longrightarrow 0.
%$$
Consider the sequence of approximate isometries $f_{i} : X \to X_{\infty}$ associated to the convergence.
For any two real numbers $\delta,\eta$  and $i \in \N$ sufficiently large it holds:
$$
\left| \frac{1}{r_{i}}  \sfd(\gamma_{r_{i}\delta}, \gamma_{r_{i}\eta}) - \sfd_{\infty}(f_{i}(\gamma_{r_{i}\delta}), f_{i}(\gamma_{r_{i}\eta}) )   \right| \leq \ve_{i}.
$$
Thanks to metric differentiability
$$
\lim_{i\to \infty} \sfd_{\infty}(f_{i}(\gamma_{r_{i}\delta}), f_{i}(\gamma_{r_{i}\eta}) ) = |\delta - \eta| |d\gamma|(0).
$$
%
%We define a sequence of curves:
%\[
%\gamma^{i} \colon \left( -\frac{c}{r_{i}}, \frac{c}{r_{i}} \right)  \to X_{\infty}, \qquad \tau \to  \gamma^{i}_{\tau} : = f_{i}(\gamma_{r_{i}\tau}). 
%\]
%Then since $\gamma^{i}_{0} = \bar x_{\infty}$ and 

Since the limit space is proper, using a diagonal argument, we have convergence of $\{ f_{i}(\gamma_{r_{i}\delta}) \}_{i \in \N}$ for all rational numbers $\delta$.  
By density and there exists a curve $z \colon \R \to X_{\infty}$ such that
$$
\sfd_{\infty} (z_{\eta},z_{\delta}) = |\eta - \delta | \cdot |d\gamma|(0).
$$
It follows that $z(\R)$ is isometric to $\R$.
\end{proof}

In Lemma \ref{L:line} we have not assumed any length structure on the metric space $(X,\sfd)$. 
Hence the assumption $\dom(\gamma) = [-c,c]$ could sound a bit restrictive.
In what follows we consider $\gamma \in \Gamma(X)$ with a more general domain.

\begin{proposition}\label{P:generaline}
Let $(X,\sfd)$ be a complete, proper and separable metric space. Fix $\bar x \in X$ and assume 
the existence of a pointed, proper, complete and separable metric space $(X_{\infty}, \sfd_{\infty},\bar x_{\infty}) \in \Tan(X,\sfd,\bar x)$.
Let $\gamma \in \Gamma(X)$ be such that
$$
\gamma_{0} = \bar x, \qquad  0 \textrm{ is a point of density one of \ } \dom(\gamma), \qquad  |d\gamma|(0)>0. 
$$
Then $(X_{\infty}, \sfd_{\infty},\bar x_{\infty})$ contains a line, limit of $\gamma$.
\end{proposition}

\begin{proof}
Let us consider fixed sequences of positive numbers $\ve_{i} \to 0$, $R_{i} \to \infty$ and 
$$
f_{i} \colon B^{X_{i}}_{R_{i}}(\bar x) \to B^{X_{\infty}}_{R_{i}}(\bar x_{\infty}),  \qquad \ve_{i}-\textrm{isometry}, \qquad f_{i}(\bar x) = \bar x_{\infty},
$$
where $B^{X_{i}}_{R_{i}}(\bar x)$ is the ball in $(X,\sfd/r_{i})$, centered in $\bar x$ and of radius $R_{i}$. 
\medskip

{\bf Step 1.} \\
Denote with $I : = \dom(\gamma)$ and for any positive $r$ we consider $I_{r}: = \{ x \in \R : x r \in I\}$.
Consider any sequence $\r_{i} \to 0$, then for each $n$ define
$$
I(n) : = \bigcup_{i \geq n} I_{\r_{i}}.
$$
The set $\cap_{n \in \N} I(n)$ is formed by all real numbers $\delta$ such that there exists a subsequence $\r_{i_{k}}$ 
so that $\r_{i_{k}}\delta  \in \dom(\gamma)$ for all $k \in \N$. 
To underline its dependence on $\{\r_{i} \}_{i \in \N}$, we also use the following notation 
$$
I(\{\r_{i} \}) : = \bigcap_{n \in \N} I(n).
$$
We observe that for each $n\in \N$ 
$$
\mathcal{L}^{1} \left( \R \setminus I(n) \right) =0.
$$
Indeed for any $M > 0$, and $j \in \N$, $j \geq n$
$$
\left( (-M,M) \setminus \bigcup_{i\geq n} I_{\r_{i}} \right)  \subset \left( (-M,M) \setminus I_{\r_{j}} \right).
$$
Then since $0$ has density one in $I$, 
$$
\lim_{i \to \infty} \frac{\mathcal{L}^{1} (\left\{  \delta \in \R \,: \, |\delta| \leq \r_{i} M, \ \delta \notin I \right\}) } {2\r_{i} M} = 0.
$$
Since 
$$
\mathcal{L}^{1} (\left\{  \delta \in \R \,: \, |\delta| \leq \r_{i} M, \  \delta \notin I \right\})  = \r_{i} \mathcal{L}^{1} (\left\{  \delta \in \R\, : \, |\delta| \leq M,\ \r_{i} \delta \notin I \right\}), 
$$
it follows that 
$$
\lim_{j \to \infty} \mathcal{L}^{1}\left( (-M,M) \setminus I_{\r_{j}} \right) = \lim_{j \to \infty} \mathcal{L}^{1}(\left\{  \delta \in \R \, : \, |\delta| \leq M,\ \r_{j} \delta \notin I \right\}) =0.
$$
Therefore for any $M \in \R$
\[
\mathcal{L}^{1}\left( (-M,M) \setminus  I(n) \right) = 0,
\]
consequently $\mathcal{L}^{1}\left( \R \setminus  I(n) \right) = 0$ for all $n \in \N$, and finally also 
\begin{equation}\label{E:points}
\mathcal{L}^{1}\left( \R \setminus  I(\{\r_{i}\}) \right) = 0,
\end{equation}
holds. 
\medskip

{\bf Step 2.} \\
Consider now the sequence of radii $\{r_{i}\}$ for which the pointed Gromov-Hausdorff convergence holds. Consider also a sequence $\eta_{n} \to 0$ and 
an enumeration of all rational numbers $\{ q_{m} \}_{m \in \N}$. Consider now a bijection $\N \ni h \to (n(h), m(h)) \in \N \times \N$ and the 
associated family of open balls in $\R$:
$$
B_{h} : = B_{\eta_{n(h)}} (q_{m(h)}).
$$
Then from \eqref{E:points}, $B_{1} \cap I(\{r_{i} \}) \neq \emptyset$. 
Hence by definition of $I(\{r_{i} \})$, there exists $t_{1} \in B_{1}$ for which there exists a subsequence of $\{r_{i} \}_{i \in \N}$, say $\{ r_{i_{1}(k)} \}_{k\in \N}$ 
so that 
$$
t_{1} r_{i_{1}(k)} \in I = \dom(\gamma), \quad \forall \, k \in \N.
$$
In particular we can consider the sequence 
$$
\{ f_{i_{1}(k)} \circ \gamma_{t_{1} r_{i_{1}(k)}} \}_{ k\in \N} \subset X_{\infty}, 
$$
where $f_{i}$ is an $\ve_{i}$-isometry from pointed Gromov-Hausdorff convergence. 
Then since the aforementioned sequence stays in a bounded neighborhood of $\bar x_{\infty}$ and $(X_{\infty}, \sfd_{\infty})$ is proper, 
there exists a subsequence of $\{t_{1} r_{i_{1}(k)}\}_{ k\in \N}$ still denoted by $\{t_{1} r_{i_{1}(k)}\}_{ k\in \N}$, 
such that
$$
z_{1} = \lim_{k \to \infty}f_{i_{1}(k)} \circ \gamma_{t_{1} r_{i_{1}(k)}},
$$
for some $z_{1} \in X_{\infty}$.

We repeat the construction now with $h =2$. Again from \eqref{E:points}, $B_{2} \cap I(\{r_{i_{1}(k)} \}) \neq \emptyset$ and therefore there exist 
$t_{2} \in B_{2}$ and a subsequence of $\{ r_{i_{1}(k)} \}_{k\in \N}$, call it $\{ r_{i_{2}(k)} \}_{k\in \N}$ for which $t_{2} r_{i_{2}(k)} \in I$ and 
$$
z_{2} = \lim_{k \to \infty}f_{i_{2}(k)} \circ \gamma_{t_{2} r_{i_{2}(k)}},
$$
for some $z_{2} \in X_{\infty}$.

Thanks to \eqref{E:points}, we can repeat the same argument for any $h$ and with a diagonal argument
we infer the existence of sequences $\{ t_{h}\}_{ h\in \N}$ and $\{ r_{i_{k}} \}_{k \in \N}$,
such that for any $h$, for all sufficiently large $k$ we have
$$
r_{i_{k}} t_{h} \in I,  \qquad z_{h} = \lim_{k \to \infty}  f_{i_{k}} \circ \gamma_{t_{h} r_{i_{k}}}.
$$

{\bf Step 3.} \\
For $n, m \in \N$ we have: 
$$
\sfd_{\infty}(z_{n},z_{m}) = \lim_{k\to \infty} \sfd_{\infty} (f_{i_{k}} \circ \gamma_{t_{n} r_{i_{k}}}, f_{i_{k}} \circ \gamma_{t_{m} r_{i_{k}}}) = 
\lim_{k\to \infty} \frac{1}{r_{i_{k}}} \sfd(\gamma_{t_{n} r_{i_{k}}}, \gamma_{t_{m} r_{i_{k}}}).
$$
Since $|d\gamma|(0) > 0$, we have
$$
\sfd_{\infty}(z_{n},z_{m}) = |t_{n} - t_{m}| |d\gamma|(0).
$$
Define therefore the curve:
$$
\gamma^{\infty} : \{ t_{h} \}_{h \in \N} \to X_{\infty}, \qquad \gamma^{\infty}_{t_{h}} : = z_{h}.
$$
Hence we have 
$$
\sfd_{\infty}(\gamma^{\infty}_{t_{n}},\gamma^{\infty}_{t_{m}}) = |t_{n} - t_{m}|\cdot |d\gamma|(0).
$$
Now observe that the set of points $\{t_{h} \}_{h \in \N}$ is dense in $\R$, indeed for each $h \in \N$ the inclusion $t_{h} \in B_{h}$ holds.
It follows that $\gamma^{\infty}$ can be extended by continuity to any $s \in \R$. So we have proved the existence of 
$$
\gamma^{\infty} : \R \to X_{\infty}, \qquad \sfd_{\infty}(\gamma^{\infty}_{t}, \gamma^{\infty}_{s}) = |t - s|\cdot |d\gamma|(0).
$$
The claim follows.
\end{proof}

\begin{remark}\label{R:simultaneuslimits}
The constructions done in the previous proof can be done simultaneously for finitely many curves. In particular suppose 
to have $\gamma^{1},\dots, \gamma^{n} \in \Gamma(X)$ such that $\gamma^{j}_{0} = \bar x$, $0$ is a 
point of density one of $\dom(\gamma^{j})$ and $|d\gamma^{j}|(0)>0$, for $j = 1,\dots, n$. 
Then there exists a dense countable set of times $\{ t_{h} \}$ and a subsequence ${i_{k}}$ such that:  
$$
z_{h}^{j} = \lim_{k\to \infty} f_{i_{k}} \circ \gamma^{j}_{t_{h}r_{i_{k}}}, \quad \sfd_{\infty}(z_{h}^{j}, z_{\eta}^{j}) = |t_{h} - t_{\eta} | \cdot  |d\gamma^{j}|(0),
$$
for any $h,\eta \in \N$ and $j = 1,\dots, n$.
\end{remark}

\medskip

%%%%%%%%%%%%%%%%%%%%%%%%%%%%%%%%%%%%%%%%%%%%%%%%%%
%%%%%%%%%%%%%%%%%%%%%%%%%%%%%%%%%%%%%%%%%%%%%%%%%%
%%%%%%%%%%%%%%%%%%%%%%%%%%%%%%%%%%%%%%%%%%%%%%%%%%
%%%%%%%%%%%%%%%%%%%%%%%%%%%%%%%%%%%%%%%%%%%%%%%%%%
%%%%%%%%%%%%%-Moving the base point----%%%%%%%%%%%%%%%%%%%%%%%%

\subsection{Change of base point}

It is also possible to extend Proposition \ref{P:generaline} to almost any other point of the tangent space that is, 
if $(X_{\infty},\sfd_{\infty}, \bar x_{\infty})$ is a pointed tangent space of $(X,\sfd, \bar x)$ 
and $z_{\infty} \in X_{\infty}$, then one can find a tangent line passing through $z_{\infty}$, obtained as the blow-up 
of the same curve. 

This can be obtained using the fact that for almost every $\bar x_\infty$ and for every $z_\infty \in X_\infty$
also $(X_{\infty},\sfd_{\infty}, z_{\infty})$ is a pointed tangent space, 
provided the ambient measure $\mm$ is doubling. This has been proved by Preiss in \cite{P1987} in the Euclidean framework 
and adapted to the metric space case by Le Donne in \cite{LD2011}. 

\begin{theorem}[\cite{LD2011}, Theorem 1.1]\label{thm:Enrico}
Let $(X,\sfd, \mm)$ be a doubling metric measure space. Then for $\mm$-a.e. $x \in X$, for all 
$(X_{\infty},\sfd_{\infty},\bar x_{\infty}) \in \Tan(X,\sfd,\bar x)$, and for all $z_{\infty} \in X_{\infty}$ we have  
$$
(X_{\infty},\sfd_{\infty}, z_{\infty}) \in \Tan(X,\sfd,\bar x).
$$
\end{theorem}

Combining Proposition \ref{P:generaline} and Theorem \ref{thm:Enrico} we have 

\begin{corollary}\label{C:differentpoint}
Let $(X,\sfd, \mm)$ be a doubling metric measure space, $\bar x \in X$ outside
the exceptional set of Theorem \ref{thm:Enrico} and 
$\gamma \in \Gamma(X)$ such that
$$
\gamma_{0} = \bar x, \qquad  0 \textrm{ is a point of density one of \ } \dom(\gamma), \qquad  |d\gamma|(0)>0. 
$$
Then for any  $(X_{\infty}, \sfd_{\infty},\bar x_{\infty}) \in \Tan(X,\sfd,\bar x)$ and any $z_{\infty} \in X_{\infty}$
there exists a line, limit of $\gamma$, passing through $z_{\infty}$.
\end{corollary}

%%%%%%%%%%%%%%%%%%%%%%%%%%%%%%%%%%%%%%%%%%%%%%%%%%%%
%%%%%%%%%%%%%%%%%%%%%%%%%%%%%%%%%%%%%%%%%%%%%%%%%%%%
%%%%%%%%%%%%%%%%%%%%%%%%%%%%%%%%%%%%%%%%%%%%%%%%%%%%
%%%%%%%%%%%%%%%%%%%%%%%%%%%%%%%%%%%%%%%%%%%%%%%%%%%%
%%%%%%%%%%%%%%%%%%%%%%%%%%%%%%%%%%%%%%%%%%%%%%%%%%%%
%-------------APPLICATIONS TO LIP DIFF SPACES---------------------------------

\section{Tangent lines in Lipschitz differentiability spaces}\label{sec:tanlipspace}
In order to apply metric differentiability to Lipschitz differentiability spaces, a Borel regularity with respect to a precise Polish space is needed. 
We therefore recall few definitions from \cite{bate:measurelip} that will be needed only in this section.

For a metric space $(X,\sfd)$ define $H(X)$ to be the collection of non-empty compact subsets of $\R\times X$ with the 
Hausdorff metric, so that $H(X)$ is complete and separable.  Moreover identify $\Gamma(X)$ with its isometric image in $H(X)$ 
via the map $\gamma \to \gr(\gamma)$ and consider 
$$
A(X) : = \left\{ (x,\gamma) \in X \times \Gamma(x)\, : \, \exists \ t \in \dom(\gamma), \ x = \gamma_{t} \right\}.
$$
One can show (Lemma 2.7, \cite{bate:measurelip}) that $\Gamma(X)$ is a Borel subset of $H(X)$ and $A(X)$ is a Borel subset of $X\times H(X)$.
Modifying Lemma 2.8 of \cite{bate:measurelip} we obtain

\begin{lemma}\label{L:Borel}
Let $(X,\sfd)$ be a complete and separable metric space. The map $F \colon A(X) \to \R \cup \{ \infty\}$
defined as
\begin{equation}\label{E:Borelldiff}
F(x,\gamma) : =
			\begin{cases}
						|d\gamma|(\gamma^{-1}(x)) 	& 	\textrm{if it exists} \\ 	
						\infty 					& 	\textrm{otherwise}	
			\end{cases}
\end{equation}
is Borel.
\end{lemma}

\begin{proof}
The proof is a slight modification of the proof of Lemma 2.8 of \cite{bate:measurelip}. \\ 
Let $q, \delta, \epsilon > 0$ and $\alpha \in (0,1]$. The set of $(\gamma_{t_{0}}, \gamma) \in A(X)$ with 
$$
\big| d(\gamma_{t_{0}+t}, \gamma_{t_{0}+ s}) - q|t-s| \big| \leq \epsilon |t-s|,
$$
for all $t,s$ with $t_{0}+t, t_{0}+s \in \dom(\gamma)$ and $|t|, |s| \leq \delta$ and $|t-s| \geq \alpha \max\{t,s\}$,  is closed.
After taking suitable countable intersection 
and unions as in \cite{bate:measurelip}, the set where $F$ belongs to some open set of $\R$ is Borel and the claim follows.
\end{proof}

\medskip

Now we obtain the following improved version of Theorem \ref{T:bate}, stated in Section \ref{Ss:diffspaces}.

\medskip

\begin{proposition}\label{P:differentiable1}
Let $(U,\f)$ be an $n$-dimensional chart in a Lipschitz differentiability space $(X,\sfd,\mm)$. Then for almost every $x \in U$, there exists 
$\gamma_{1}^{x},\dots, \gamma_{n}^{x} \in \Gamma(X)$ such that for each $i =1, \dots, n$
\begin{itemize}
\item[i)] $(\gamma_{i}^{x})^{-1}(x) = 0$ is a point of density one of $(\gamma_{i}^{x})^{-1}(U)$; 
\item[ii)] the metric differential in $0$ exists and $|d\gamma^{x}_{i} |(0)>0$;
\item[iii)]$(\f \circ \gamma_{i}^{x})'(0)$ are linearly independent.
\end{itemize}
Moreover, for any such $\gamma_{i}^{x}$, for any Lipschitz $g \colon X \to \R$ and almost every $x \in U$, the gradient of $g$ at $x$ with respect to $\f$ 
and $\gamma_{1}^{x},\dots, \gamma_{n}^{x}$ equals $Dg(x)$, 
that is 
$$
\left(g \circ \gamma_{i}^{x}\right)'(0) = Dg (x_{0})\cdot \left(\f\circ \gamma_{i}^{x}\right)'(0),
$$
for $i =1,\dots, n$.
\end{proposition}

Even though the proof of Proposition \ref{P:differentiable1} contains no novelty with respect to Theorem \ref{T:bate}, we included it here for reader's convenience. 

\begin{proof}
By Theorem 6.6 of \cite{bate:measurelip} we have the existence of a countable decomposition $U = \cup_{j} U_{j}$ of $U$ into sets with $n$ 
$\f$-independent Alberti representations (whose definition can be found in Section 2 of \cite{bate:measurelip}).
We consider for $k = 1, \dots, n$ the Borel function $F_{k}: A(X) \to \R \cup \{ \infty\}$ defined by
$$
F_{k}(x,\gamma) : =
			\begin{cases}
						(\f^{k} \circ \gamma)'(\gamma^{-1}(x)) 	& 	\textrm{if it exists} \\ 	
						\infty 							& 	\textrm{otherwise},	
			\end{cases}
$$
where $\f^{k}$ is the $k$-th component of the coordinate map $\f : U \to \R^{n}$. 
Moreover, we define $F_{0}$ to be the function $F$ considered in Lemma \ref{L:Borel}.

For each $k = 0, \dots, n$ all the assumption of Proposition 2.9 of \cite{bate:measurelip} are satisfied. 
The case $k = 0$ follows from Proposition \ref{P:metricdiff} and Lemma \ref{L:Borel}, while the case $k \geq 1$ from Lemma 2.8 of \cite{bate:measurelip}.
 Then we can repeat
the same argument for the $n$ $\f$-independent Alberti representations on each $U_{j}$.

Hence for each $j \in \N$ there exists $V_{j} \subset U_{j}$ with $\mm(U_{j} \setminus V_{j}) = 0$ such that for each $x \in V_{j}$
there exist $\gamma^{x}_{1}, \dots, \gamma^{x}_{n} \in \Gamma(X)$ such that
\begin{itemize}
\item[i)] $(\gamma_{i}^{x})^{-1}(x) = 0$ is a point of density one of $(\gamma_{i}^{x})^{-1}(V_{j})$; 
\item[ii)] the metric differential in $0$ exists and $|d\gamma^{x}_{i} |(0)>0$;
\item[iii)]$(\f \circ \gamma_{i}^{x})'(0)$ are linearly independent,
\end{itemize}
for $i =1,\dots, n$ and for each $k = 0, \dots, n$ the map $x \mapsto F_{k}(x,\gamma^{x}_{i})$ is measurable. 
Since $V_{j} \subset U$, i) implies that $(\gamma_{i}^{x})^{-1}(x) = 0$ is a point of density one of $(\gamma_{i}^{x})^{-1}(U)$. 
This proves the first part of the statement. The second part just follows from Theorem \ref{T:bate}. 
\end{proof}

We can now use the previous result to obtain the following 

\begin{proposition}\label{P:result1}
Let $(X,\sfd,\mm)$ be a Lipschitz differentiability space and $(U,\f)$ be an $n$-dimensional chart. 
Then for $\mm$-almost every $\bar x \in U$, any element $(X_{\infty},\sfd_{\infty},\bar x_{\infty}) \in \Tan(X,\sfd,\bar x)$ contains $n$ disjoint 
(neglecting $\bar x_{\infty}$) isometric copies of $\R$, obtained as limits of Lipschitz curves.
\end{proposition}

\begin{proof} Take any pointed metric measure space $(X_{\infty},\sfd_{\infty},\bar x_{\infty}) \in \Tan(X,\sfd,\bar x)$ 
and the corresponding sequence of dilations $r_{i}> 0$, with $r_{i} \to 0$.

The existence of $n$ isometric copies of $\R$ follows straightforwardly from Definition \ref{D:pmGH}, Proposition \ref{P:generaline} and Proposition \ref{P:differentiable1}.
It only remains to prove that the copies are disjoint. 
To this end we consider a chart $(U,\f)$ with $\bar x \in U$ and use Remark \ref{R:simultaneuslimits}: 
there exists a dense sequence  $\{ t_{h} \}_{h\in \N}\subset \R$ and a subsequence ${i_{k}}$ such that:  
$$
z_{h}^{j} = \lim_{k\to \infty} f_{i_{k}} \circ \gamma^{j}_{t_{h}r_{i_{k}}}, \quad \sfd_{\infty}(z_{h}^{j}, z_{\eta}^{j}) = |t_{h} - t_{\eta} | \cdot  |d\gamma^{j}|(0),
$$
for $j = 1, \dots, n$, where $f_{i_{k}}$ is the sequence of approximate isometries and $\gamma^{j}$ are given by Proposition \ref{P:differentiable1}. 
Recall that the closure in $\sfd_{\infty}$ of each $\{ z_{h}^{j} : h \in \N \}$ forms the isometric copies of $\R$ in $X_{\infty}$.  Via a reparametrization, without loss of generality, 
we may also assume that $|d\gamma^{j}|(0) = 1$ for all $ j = 1,\dots, n$.

Now we just observe that for $j ,l = 1, \dots, n$
\begin{align*}
\sfd_{\infty} ( z_{h}^{j}, z_{h}^{l} ) 
	= 	&~ \lim_{k \to \infty }   \sfd_{\infty} ( f_{i_{k}} \circ \gamma^{j}_{t_{h}r_{i_{k}}}, f_{i_{k}} \circ \gamma^{l}_{t_{h}r_{i_{k}}} ) \crcr
	= 	&~ \lim_{k \to \infty }   \frac{1}{r_{i_{k}}} \sfd ( \gamma^{j}_{t_{h}r_{i_{k}}}, \gamma^{l}_{t_{h}r_{i_{k}}} ) \crcr
	\geq &~ \frac{1}{L} \lim_{k \to \infty }   \frac{1}{r_{i_{k}}} |    \f \circ \gamma^{j}_{t_{h}r_{i_{k}}}-   \f \circ \gamma^{l}_{t_{h}r_{i_{k}}} |\crcr
	\geq &~ \frac{t_{h}}{L}  |   \left( \f \circ \gamma^{j} \right)'(0) -   \left( \f \circ \gamma^{l} \right)'(0) |,
\end{align*}
where $L$ is the Lipschitz constant of $\f$. 
Therefore we have proved that 
\begin{equation}\label{E:independence}
\sfd_{\infty} ( z_{h}^{j}, z_{h}^{l} ) \geq \frac{t_{h}}{L}  |   \left( \f \circ \gamma^{j} \right)'(0) -   \left( \f \circ \gamma^{l} \right)'(0) |,
\end{equation}
that implies, by linear independence, that $\sfd_{\infty} ( z_{h}^{j}, z_{h}^{l} ) > 0$, for all $h \in \N$. 
Since intersection for different times is not possible (at time 0 they start from the same point, with the same speed), the claim follows. 
\end{proof}

We summarize the disjointness property of the isometric embeddings of $\R$.

\begin{corollary}\label{C:resume}
Let $(X,\sfd,\mm)$ be a Lipschitz differentiability space and $(U,\f)$ be an $n$-dimensional chart. 
Then for $\mm$-almost every $\bar x \in U$, there exist $v_{1}, \dots, v_{n} \in \R^{n}$ linearly independent such that 
for any element $(X_{\infty},\sfd_{\infty},\bar x_{\infty}) \in \Tan(X, \sfd,\bar x)$  there exist 
$\iota_{1}, \dots, \iota_{n} \colon \R \to X_{\infty}$ so that  \medskip
\begin{itemize}
	\item[i)] 	$\iota_{j}(0) = \bar x_{\infty}$, for any $j = 1,\dots, n$; \\
	\item[ii)] 	$\sfd_{\infty}(\iota_{j}(t),\iota_{j}(s)) = |t - s|$, for any $j = 1,\dots, n$, for all $s,t \in \R$; \\
	\item[iii)] 	$\sfd_{\infty}(\iota_{j}(t),\iota_{k}(t)) \geq C |t| \cdot |v_{j} - v_{k}|$, for any $j, k = 1,\dots, n$, for all $t \in \R$; \medskip
\end{itemize}
for some positive constant $C$. Each of the $\iota_{i}$ is obtained as the limit of a Lipschitz curve.
\end{corollary}

If the Lipschitz differentiability space is also doubling, one can argue as in Corollary \ref{C:differentpoint} 
to obtain information on lines through any point of the tangent space.

\begin{theorem}\label{T:curves}
Let $(X,\sfd,\mm)$ be a doubling Lipschitz differentiability space and $(U,\f)$ be an $n$-dimensional chart. 
Then for $\mm$-almost every $\bar x \in U$, there exist $v_{1}, \dots, v_{n} \in \R^{n}$ linearly independent such that 
for any element $(X_{\infty},\sfd_{\infty},\bar x_{\infty}) \in \Tan(X, \sfd,\bar x)$  and 
for each $z \in X_{\infty}$ there exist 
$\iota^{z}_{1}, \dots, \iota^{z}_{n} \colon \R \to X_{\infty}$ so that \medskip
\begin{itemize}
	\item[i)] 	$\iota^{z}_{j}(0) = z$, for any $j = 1,\dots, n$; \\
	\item[ii)] 	$\sfd_{\infty}(\iota^{z}_{j}(t),\iota^{z}_{j}(s)) = |t - s|$, for any $j = 1,\dots, n$, for all $s,t \in \R$; \\
	\item[iii)] 	$\sfd_{\infty}(\iota^{z}_{j}(t),\iota^{z}_{k}(t)) \geq C |t| \cdot |v_{j} - v_{k}|$, for any $j, k = 1,\dots, n$, for all $t \in \R$; \medskip
\end{itemize}
for some positive constant $C$. For each $z \in X_{\infty}$, each line $\iota^{z}_{i}$ is obtained as the blow-up of a Lipschitz curve, with the blow-up depending on $z$.
\end{theorem}

%\begin{corollary}\label{C:surjcetivechart}
%Let $(X,\sfd,\mm)$ be a Lipschitz differentiability space and $(U,\f)$ be an $n$-dimensional chart. 
%Assume moreover $\mm$ to be doubling. Then any tangent function $u_{\f}$ compatible
%\end{corollary}
%

%%%%%%%%%%%%%%%%%%%%%%%%%%%%%%%%%%%%%%%%%%%%%%%%%%%%%%
%%%%%%%%%%%%%%%%%%%%%%%%%%%%%%%%%%%%%%%%%%%%%%%%%%%%%%
%%%%%%%%%%%%%%%%%%%%%%%%%%%%%%%%%%%%%%%%%%%%%%%%%%%%%%
%%%%%%%%%%%%%%%%%%%%%%%%%%%%%%%%%%%%%%%%%%%%%%%%%%%%%%
%%%%%%%%%%%%%%%%%%%%%%%%%%%%%%%%%%%%%%%%%%%%%%%%%%%%%%
%%----------------------APPLICATION TO REGULAR SPACES-----------------%%%%%%%%%%%%%%%%%%

\section{Tangent lines in spaces with splitting tangents}\label{sec:split}

As stated in Theorem \ref{T:cheeger}, doubling metric measure spaces supporting a local
$p$-Poincar\'e inequality are Lipschitz differentiability spaces
and in particular Corollary \ref{C:resume} applies. 
For this class of more regular metric measure spaces, results on the structure of tangent spaces were already at disposal. 
For instance in \cite{cheeger:lip}, Theorem 8.5, existence of integral curves for tangent functions was proved.
This in turn implies the existence of sufficiently many geodesic lines in the tangent space. 
But no explicit relation between geodesic lines in the tangent space and curves on the metric measure space was shown to exist.
Therefore Corollary \ref{C:resume} brings new information also on the structure of tangent 
spaces for doubling metric measure space supporting a local $p$-Poincar\'e inequality.

In this last section we show that if $n$ is the dimension of a chart of the measurable differentiable structure of $(X,\sfd,\mm)$ seen 
as a Lipschitz differentiability space and if $d$ is the dimension of a Euclidean tangent space at $x$, then $n\leq d$ at $\mm$-a.e. point of $X$.
In particular, we are interested in a special class of metric measure spaces $(X,\sfd,\mm)$ 
having the \emph{splitting of tangents} property: if
$$
 (X_{\infty},\sfd_{\infty},\bar x_{\infty}) \in \Tan(X,\sfd,\bar x)
$$
for some $\bar x \in X$ and if $X_{\infty}$ contains an isometric copy of $\R$ going through $\bar x_\infty$,
then $(X_{\infty},\sfd_{\infty})$ is isometric to
$$
 (\R \times Y,|\cdot| \times \sfd_Y)
$$
where $(Y,\sfd_Y)$ is a metric space.

We obtain the following result.

\begin{theorem}\label{T:n=k}
 Suppose that $(X,\sfd,\mm)$ is a doubling Lipschitz differentiability space with the splitting of tangents property.
 Let $(U,\f)$ be an $n$-dimensional chart of $(X,\sfd,\mm)$. 
 Then for $\mm$-a.e. $\bar x \in U$  any
 $(X_{\infty}, \sfd_{\infty},\bar x_{\infty}) \in \Tan(X,\sfd,\bar x)$ is of the form
$$
(X_{\infty}^{d}\times \R^{d},\sfd^{d}_{\infty}\times | \cdot |_, (\bar x^{d}_{\infty}, 0)),
$$
with $d\geq n$. 
%Moreover 
%%
%$$
%\iota_{j}(\R) \subset \{ (x^{n}_{\infty}, v ) : v \in \R^{n} \}, \qquad \forall \, j = 1, \dots, n,
%$$
%%
%and each $\iota_{j}$ is obtained as the limit of Lipschitz curve.
\end{theorem}

Compare Theorem \ref{T:n=k}  to the result from \cite{GMR}, that 
can be rephrased as
\begin{theorem}\label{thm:gmr}
 Suppose that $(X,\sfd,\mm)$ is a geodesic doubling metric measure space with the splitting of tangents property.
 Then at $\mm$-a.e. point in $X$ there exists a Euclidean tangent space.
\end{theorem}

Theorem \ref{thm:gmr} was formulated in \cite{GMR} for $\mathsf{RCD}^{*}(K,N)$ spaces (metric measure spaces
with Riemannian Ricci curvature bounded below by $K \in \R$ and dimension from above by $N$), for which any tangent
is an $\mathsf{RCD}^{*}(0,N)$ space having the splitting property, as was shown by Gigli \cite{gigli:split}.
Theorem \ref{T:n=k} now shows that taking into account the fact that $\mathsf{RCD}^{*}(K,N)$ spaces are doubling
and support a local Poincar\'e inequality \cite{R2012a,R2012b}, we immediately have that any tangent space contains
an $\R^n$ part with dimension at least the dimension of the chart. 
For a comprehensive treatise on the above mentioned family of spaces we refer
to \cite{lottvillani:metric,sturm:I, sturm:II} for the defintion of $\mathsf{CD}(K,N)$ and to 
\cite{AGS, AGS11b} for the infinite dimensional Riemannian version. 
Finally $\mathsf{RCD}^{*}(K,N)$ with $N \in \R$ has been introduced independently in \cite{AMS} and \cite{EKS}.

For $\mathsf{RCD}^{*}(K,N)$ spaces more
can be said on the relation of the charts and the tangent spaces
than the conclusion of Theorem \ref{T:n=k}.
A recent result by Mondino and Naber in \cite{mondino:tangent} states that for 
$(X,\sfd,\mm)$ verifying $\mathsf{RCD}^{*}(K,N)$, at $\mm$-a.e. $x \in X$ 
there exists a unique tangent space and it is isomorphic, in the sense of metric measure spaces, 
to  $(\R^{d}, |\cdot|, \mathcal{L}^{d})$, with $d$ varying measurably in $x$. 
Moreover, they proved the following theorem.
\begin{theorem}\label{thm:M-N}
Let $(X,\sfd,\mm)$ be an $\mathsf{RCD}^{*}(K,N)$ space for some $K,N \in \R$ with $N > 1$.
Then there exists a countable collection $\{R_j\}_{j \in \N}$ of $\mm$-measurable subsets of $X$,
covering $X$ up to an $\mm$-negligible set, such that each $R_j$ is biLipschitz to a measurable subset of $\R^{k_j}$,
for some $1 \le k_j \le N$, $k_j$ possibly depending on $j$.
\end{theorem}

Combining this result with the fact that if a Lipschitz differentiability space is (locally) biLipschitz embeddable into a Euclidean space,
then at almost every point all the tangent spaces are biLipschitz equivalent to $\R^n$, where $n$ is the dimension 
of the chart, see Corollary 8.3 in \cite{david}. Therefore, for $\mathsf{RCD}^{*}(K,N)$ spaces at almost every point
the tangent is $\R^n$ where the $n$ is the dimension of the chart. Let us note that it is still unknown if in this context the
dimension $n$ of the tangent (and the chart) depends on the point.\\

We prove Theorem \ref{T:n=k}, which is valid without the biLipschitz embeddability to $\R^n$. 

\medskip

\begin{proof}[Proof of Theorem \ref{T:n=k}]

{\bf Step 1.} \\
\noindent
By Corollary \ref{C:resume}
any element $(X_{\infty},\sfd_{\infty},\mm_{\infty},\bar x_{\infty}) \in \Tan(X,\sfd,\mm,\bar x)$ have $n$ distinct isometric copies of $\R$: 
$$
\iota_{j} \colon \R \to X_{\infty}, \qquad \iota_{j}(0) = \bar x_{\infty}, \qquad j = 1,\dots, n.
$$
By the splitting property, there exists an isometry 
$$
h_{1} \colon (X_{\infty}, \sfd_{\infty})  \longrightarrow (X_{\infty}^{1} \times \R , \sfd_{\infty}^{1} \times |\cdot|),  \qquad h_{1}(\bar x_{\infty} ) = (\bar x_{\infty}^{1},0),  
$$
with $h_{1}(\iota_{1}(\R)) =\{ (\bar x_{\infty}^{1},t) : t\in \R \}$ and splitting the measure.
Since the $n$ geodesics are all disjoint, composing isometries and applying Lemma \ref{L:geodesicproduct} we deduce the existence of $n-1$ geodesics, again denoted with 
$$
\iota_{j} : (\R,|\cdot|) \to (X_{\infty}^{1}, \sfd^{1}_{\infty}), \qquad j = 2, \dots, n.
$$
By Lemma \ref{L:geodesicproduct} we can also deduce that $\iota_{2}(\R), \dots, \iota_{n}(\R)$ are all disjoint 
and we can use again the splitting property to rule out another isometric copy of $\R$. 

The same reasoning cannot be repeated to obtain a splitting of the form $X_{\infty} \sim X^{n}_{\infty} \times \R^{n}$. It might be the case that  
for some $j = 3, \dots, n$, $\iota_{j}(\R)$ is already contained in the Euclidean component of the tangent space, and therefore the projection in the purely 
metric component of $X_{\infty}$ could be the constant geodesic, not producing a new component to rule out via the splitting property.  

\medskip

{\bf Step 2.} \\
\noindent
Consider the $n$-dimensional chart $(U,\f)$  with $\f : U \to \R^{n}$ Lipschitz and any $\bar x \in U$ such that Corollary \ref{C:differentpoint} applies.
Fix also $(X_{\infty},\sfd_{\infty},\mm_{\infty}, \bar x_{\infty}) \in \Tan(X,\sfd,\mm,\bar x)$ and $u_{\f}$, the tangent function of $\f$ at $\bar x$. 
Note that, possibly passing to subsequences, $u_{\f}$ is well-defined.

Repeating the argument of {\bf Step 1.} changing the reference point (see \cite{LD2011}, Theorem 1.1), we have the following:
for some $d \in \N$ all the possible splittings obtained from the lines of Theorem \ref{T:curves} give a decomposition of the following type: 
$X_{\infty} =  X^{d}_{\infty} \times \R^{d}$, where the identity holds in the sense of metric measure spaces, and $\R^{d}$ is equipped with the Euclidean distance 
and the Lebesgue measure.

\medskip

{\bf Step 3.} \\ 
\noindent
Consider the sequence $\{r_{i} \}_{i\in \N}$ producing $(X_{\infty},\sfd_{\infty},\mm_{\infty}, \bar x_{\infty})$ as the tangent space  
and the $\ve_{i}$-isometries $f_{i}$ and $f^{-1}_{i}$. 
Let $z \in X_{\infty}$ be any point, then by definition 
$$
u_{\f}(z) = \lim_{i \to \infty} \frac{\f(f^{-1}_{i} (z)  )   - \f(\bar x)}{r_{i}}. 
$$
As observed in the proof of Proposition \ref{P:result1}, after a suitable reparametrization with unit speed, 
there exists a sequence of times $\{t_{i}\}_{i \in \N}$ with $t_{i} \to 1$ as $i \to \infty$ such that
$\sfd_{\infty} (\iota^{0}_{j}(1), f_{i} (\gamma^{j}_{t_{i}r_{i}})) \to 0$, for each $j = 1, \dots, n$. We pose $z = \iota^{0}_{j}(1)$
and observe that 
\begin{align*}
\frac{1}{r_{i}} | \f(\gamma^{j}_{t_{i}r_{i}})  -  \f(  f^{-1}_{i}(z)  )  | 
		\leq 	&~  L \frac{1}{r_{i}} \sfd(\gamma^{j}_{t_{i}r_{i}}, f^{-1}_{i}(z)  ) \crcr
		\leq 	&~ L \left( \ve_{i }  + \sfd_{\infty} (f_{i}(\gamma^{j}_{t_{i}r_{i}}) , f_{i} (f^{-1}_{i}(z))  ) \right) \crcr
		\leq 	&~ L \left( \ve_{i }  + \sfd_{\infty} (f_{i}(\gamma^{j}_{t_{i}r_{i}}) , z) + \sfd_{\infty}(z, f_{i} (f^{-1}_{i}(z))  ) \right) \crcr
		\leq 	&~ C \ve_{i}.
\end{align*}
It therefore follows that 
$$
u_{\f}(\iota^{0}_{j}(1)) = \lim_{i \to \infty} \frac{ \f(\gamma^{j}_{t_{i} r_{i} })  - \f(\bar x)}{r_{i}} = (\f \circ \gamma^{j})'(0).
$$
Using a different $t_{i}$ converging to some other real numbers, 
it is easy to observe that $s \mapsto u_{\f}(\iota^{0}_{j}(s))$ is linear and
$$
Span \left\{ u_{\f}(\iota^{0}_{1}(1)), \dots, u_{\f}(\iota^{0}_{n}(1)) \right\} = \R^{n}.
$$
Thanks to Proposition 3.1 of \cite{david}, the same argument works for any $z \in X_{\infty}$ and therefore 
$s \mapsto u_{\f}(\iota^{z}_{j}(s))$, for any $z \in X_{\infty}$ and $j = 1, \dots, n$. 
%We now recall that $u_{\f}$ is generalized linear (Section \ref{Ss:tangetfunct}).  We can therefore use Lemma \ref{L:linear}
%and deduce that it is linear on $\R^{k}$. 
%Moreover $\{\iota_{j}(1) : j = 1,\dots, n \} \subset \R^{k}$. It follows that 
%

%
%
%%
%$$
%u_{\f} : X^{d}_{\infty} \times \R^{d} \to \R^{n}
%$$
%%
%with $u_{\f}$ Lipschitz. 
Finally, we consider $\bar u_{\f}$, the restriction of $u_{\f}$ to $\left\{ \bar x_{\infty}^{d} \right\} \times \R^{d} \to \R^{n}$. 
The claim can now be proven via showing that $\bar u_{\f}$ is a quotient map (again we refer to \cite{david} for the relative definition). 
This can be obtained repeating verbatim the proof of Corollary 5.1 of \cite{david} and using the linearity of $s \mapsto u_{\f}(\iota^{0}_{j}(s))$, 
together with Theorem \ref{T:curves}.
\end{proof}


\begin{thebibliography}{10}


\bibitem{AmbrosioGigliMondinoRajala12} L.~Ambrosio, N.~Gigli, A.~Mondino, and T.~Rajala,
\newblock Riemannian {R}icci curvature lower bounds in metric measure spaces with $\sigma$-finite measure,
\newblock{\em to appear in Trans. Amer. Math. Soc.}, arxiv:1207.4924.%, 2012.

\bibitem{AGS} L.~Ambrosio, N~Gigli and G.~Savar\'e,
\newblock Bakry-\'Emery curvature-dimension condition and Riemannian Ricci curvature bounds,
\newblock{\em to appear in Annals of Probab.}.


\bibitem{AGS11b}L.~Ambrosio, N~Gigli and G.~Savar\'e,
 \newblock Metric measure spaces with {R}iemannian {R}icci curvature bounded from below, 
 \newblock{\em Duke Math. J.} \textbf{163} (2014), 1405--1490.  

\bibitem{AK2000} L.~Ambrosio and B.~Kirchheim,
 \newblock Rectifiable sets in metric and {B}anach spaces,
 \newblock{\em Math. Ann.} \textbf{318} (2000), 527--555.
 
\bibitem{AMS} L.~Ambrosio, A.~Mondino and G.~Savar\'e,  
\newblock Nonlinear diffusion equations and curvature conditions in metric measure spaces, 
\newblock {\em in preparation}.


\bibitem{ambrtilli:metricanal} L.~Ambrosio and P.~Tilli,
\newblock Topics on Analysis in Metric Spaces.
\newblock{\em Oxford University press}, Oxford Lecture Series in Mathematics and Its Applications, 2004.


\bibitem{bate:measurelip} D.~Bate,
\newblock Structure of measures in Lipschitz differentiability spaces, 
\newblock {\em Journal Amer. Math. Soc.}, to appear, arXiv 1208.1954.

\bibitem{BL} D.~Bate and S.~Li,
\newblock Characterizations of rectifiable metric measure spaces,
\newblock{\em preprint}, arXiv:1409.4242.

\bibitem{BB} M.~Bourdon and H.~Pajot,
\newblock Poincar\'e inequalities and quasiconformal structure on the boundary of some hyperbolic buildings,
\newblock{\em Proc. Amer. Math. Soc.} \textbf{127} (1999), no. 8, 2315--2324. 

\bibitem{cheeger:lip}  J.~Cheeger,
\newblock  Differentiability of Lipschitz functions on metric measure spaces,
\newblock{\em Geom. Funct. Anal.} \textbf{9} (1999), 428--517.


\bibitem{CKS}  J.~Cheeger and B.~Kleiner and A.~Schioppa,
\newblock Infinitesimal structure of differentiability spaces, and metric differentiation,
\newblock{\em preprint} arXiv:1503.07348.




\bibitem{david} G.C.~David,
\newblock Tangents and rectifiability of Ahlfors regular Lipschitz differentiability spaces,
\newblock{\em Geom. Funct. Anal.} (to appear), arXiv:1405.2461.%, (2014).

\bibitem{HK} J.~Heinonen and P.~Koskela,
\newblock Quasiconformal maps in metric spaces with controlled geometry,
\newblock{\em Acta Math.} \textbf{181} (1998), 1--61.

\bibitem{EKS} M~Erbar, ~Kuwada and K.T.~Sturm,
\newblock On the Equivalence of the Entropic Curvature-Dimension Condition and Bochner's Inequality on Metric Measure Space,
\newblock {\em Invent. Math.} (to appear) arXiv:1303.4382. %, (2013). 

\bibitem{gigli:split} N.~Gigli,
\newblock The splitting theorem in non-smooth context, 
\newblock{\em preprint}, arXiv:1302.5555.

\bibitem{GMR}  N.~Gigli, A.~Mondino and T.~Rajala,
\newblock  Euclidean spaces as weak tangents of infinitesimally Hilbertian metric measure spaces with Ricci curvature bounded below,
\newblock{\em J. Reine Angew. Math.} (to appear).% DOI 10.1515/ crelle-2013-0052. 

\bibitem{Juillet2009} N. Juillet,
\newblock Geometric inequalities and generalized Ricci bounds in the Heisenberg group,
\newblock{\em Int. Math. Res. Notices} \textbf{2009} (2009), 2347--2373.

\bibitem{keith:diff}  S.~Keith,
\newblock  Measurable Differentiable Structures and the Poincar\'e Inequality,
\newblock{\em Indiana Univ. Math. J.} \textbf{53} (2004), 1127--1150.

\bibitem{KR} C.~Ketterer and T.~Rajala,
\newblock Failure of topological rigidity results for the measure contraction property,
\newblock{\em Potential Anal.} (to appear).

\bibitem{kirch:diff}  B.~Kirchheim,
\newblock  Rectifiable metric space: local structure and regularity of the Hausdorff measure,
\newblock{\em Proc. Am. Math. Soc.} \textbf{121} (1994), 113--123.

\bibitem{korte} R.~Korte,
\newblock Geometric implications of the Poincar\'e inequality,
\newblock{\em Result. Math.} \textbf{50} (2007), 93--107.

\bibitem{laakso} T.~Laakso,
\newblock Ahlfors Q-regular spaces with arbitrary $Q>1$ admitting weak Poincar\'e inequality,
\newblock{\em Geom. Funct. Anal.} \textbf{10} (2000), 111--123.

\bibitem{LD2011} E.~Le~Donne,
\newblock Metric spaces with unique tangents,
\newblock{\em Ann. Acad. Sci. Fenn. Math.} \textbf{36} (2011), 683--694.

\bibitem{lottvillani:metric} J.~Lott and C.~Villani,
\newblock Ricci curvature for metric-measure spaces via optimal transport,
\newblock{\em Ann. of Math.} (2) \textbf{169} (2009), 903--991.

\bibitem{mondino:tangent} A.~Mondino and A.~Naber, 
\newblock Structure Theory of Metric-Measure Spaces with Lower Ricci Curvature Bounds I,
\newblock{\em preprint}, arXiv:1405.2222.

\bibitem{Ohta2007} S.-I.~Ohta,
\newblock On the measure contraction property of metric measure spaces,
\newblock{\em Comment. Math. Helv.} \textbf{82} (2007), 805--828.

\bibitem{Ohta} S.-I.~Ohta,
\newblock Splitting theorems for Finsler manifolds of nonnegative Ricci curvature,
\newblock{\em J. Reine Angew. Math.} (to appear).

\bibitem{P1987} D.~Preiss,
\newblock Geometry of measures in $\R^n$: distribution, rectifiability, and densities,
\newblock{\em Ann. of Math.} \textbf{125} (1987), 537--643.

\bibitem{R2012a} T.~Rajala,
 \newblock Local Poincar\'e inequalities from stable curvature conditions on metric spaces,
 \newblock {\em Calc. Var. Partial Differential Equations} \textbf{44} (2012), 477--494.

\bibitem{R2012b} T.~Rajala,
\newblock Interpolated measures with bounded density in metric spaces satisfying the curvature-dimension conditions of Sturm,
 \newblock {\em J. Funct. Anal.} \textbf{263} (2012), 896--924.

\bibitem{rudin:real} W.~Rudin, 
\newblock Real and complex analysis,
\newblock{\em McGraw-Hill Book Company}, International Edition 1987.

\bibitem{schioppa} A.~Schioppa
\newblock Derivations and {A}lberti representations, 
\newblock{\em preprint}, arXiv:1311.2439.

\bibitem{semmes2} S.~Semmes,
\newblock Finding curves on general spaces through quantitative topology, with applications to Sobolev and Poincar\'e inequalities,
\newblock{\em Selecta Math.} \textbf{2} (1996), 155--295. 

\bibitem{semmes} S.~Semmes,
\newblock Some novel types of fractal geometry,
\newblock{\em Oxford Mathematical Monographs. The Clarendon Press, Oxford University Press}, New York, 2001.

\bibitem{sturm:I} K.T.~Sturm, 
\newblock On the geometry of metric measure spaces. I,
\newblock{\em Acta Math.} \textbf{196} (2006), 65--131.

\bibitem{sturm:II} K.T.~Sturm, 
\newblock On the geometry of metric measure spaces. II,
\newblock{\em Acta Math.} \textbf{196} (2006), 133--177.


\end{thebibliography}
\end{document}